\documentclass[12pt]{article}
\usepackage{amsmath,amsthm,amsfonts,amssymb}
\usepackage[pdftex,pdfborder={0 0 0}]{hyperref}
\usepackage{fullpage}
\usepackage{multirow}
\usepackage{array}
\usepackage{bbm}
\usepackage[noblocks]{authblk}
\usepackage{verbatim}
\usepackage{tikz}
\usepackage{float}
\usepackage{enumerate}
\usepackage{diagbox}
\usepackage{comment}
\usetikzlibrary{arrows}
\usepackage{soul}
\usepackage{mathtools}
\usepackage{tikz}
\usepackage{longtable}
\usepackage{circuitikz}
\usepackage{markdown}
\usetikzlibrary{decorations.pathreplacing,angles,quotes,calligraphy}
\usepackage{bm}
\usepackage{subfigure}
\usepackage{caption}
\usepackage{subcaption}
\usepackage{listings}
\allowdisplaybreaks

\newtheorem{theorem}{Theorem}[section]
\newtheorem{lemma}[theorem]{Lemma}
\newtheorem{proposition}[theorem]{Proposition}
\newtheorem{corollary}[theorem]{Corollary}

\theoremstyle{definition}
\newtheorem{definition}[theorem]{Definition}

\newtheorem{remark}[theorem]{Remark}

\newlength{\Oldarrayrulewidth}

\newcommand{\floor}[1]{\left\lfloor #1 \right\rfloor}

\newcommand{\N}{\mathbb{N}}

\newcommand{\mex}{\textup{\rm MEX}}

\renewcommand{\mod}[2]{\equiv#1\textup{ (mod }#2\textup{)}}

\newcommand{\Nim}{\mathcal{G}}

\def\m@th{\mathsurround=0pt}
\def\sm#1{\null\,\vcenter{\baselineskip9pt\lineskip.23ex\m@th
    \ialign{\hfil$\scriptstyle##$\hfil&&\ \hfil$\scriptstyle##$\hfil\crcr
    \mathstrut\crcr\noalign{\kern-\baselineskip}
    #1\crcr\mathstrut\crcr\noalign{\kern-\baselineskip}}}\,}
\def\smnp#1{\null\,\vcenter{\baselineskip9pt\lineskip.23ex\m@th
    \ialign{\hfil$\scriptstyle##$\hfil&&\ \ \hfil$\scriptstyle##$\hfil\crcr
    \mathstrut\crcr\noalign{\kern-\baselineskip}
    #1\crcr\mathstrut\crcr\noalign{\kern-\baselineskip}}}\,}

\begin{document}

\title{On the Nature and Complexity of an Impartial Two-Player Variant of the Game Lights-Out\texttrademark}
\author[1]{Eugene Fiorini\thanks{efiorini@dimacs.rutgers.edu}}
\author[2]{Maxwell Fogler\thanks{mafo1356@colorado.edu}}
\author[3]{Katherine Levandosky\thanks{levandosky.k@northeastern.edu}}
\author[4]{Bryan Lu\thanks{byl29@cornell.edu}}
\author[5]{Jacob Porter\thanks{porterjk@lafayette.edu}}
\author[6]{Andrew Woldar\thanks{andrew.woldar@villanova.edu}}

\affil[1]{DIMACS-Rutgers University, Piscataway, NJ, USA}
\affil[2]{University of Colorado, Boulder, CO, USA}
\affil[3]{Northeastern University, Boston, MA, USA}
\affil[4]{Cornell University, Ithaca, NY, USA}
\affil[5]{Lafayette College, Easton, PA, USA}
\affil[6]{Villanova University, Villanova, PA, USA}

\date{}

\maketitle
 
\begin{abstract}
In this paper we study a variant of the solitaire game Lights-Out\texttrademark, where the player's goal is to turn off a grid of lights. This variant is a two-player impartial game where the goal is to make the final valid move. This version is playable on any simple graph where each node is given an assignment of either a 0 (representing a light that is off) or 1 (representing a light that is on). We focus on finding the Nimbers of this game on grid graphs and generalized Petersen graphs. We utilize a recursive algorithm to compute the Nimbers for $2 \times n$ grid graphs and for some generalized Petersen graphs. \\
\\
 \end{abstract}
\thanks{This research was partially supported by the National Science Foundation, grant numbers DMS-1852378 and DMS-2150299.} 


\section{Introduction}
Lights-Out\texttrademark \! is a commercial game that consists of turning lighted buttons on or off on a $5 \times 5$ array by pressing them one at a time. This can be represented as a $5 \times 5$ lattice with vertex labels $1$ (on) and $0$ (off). A \emph{move} involves switching the $0/1$ status of a vertex as well as the $0/1$ status of all its neighbors. A complete strategy for this game is detailed by Anderson and Feil in \cite{AF98}. In this paper we generalize this concept and consider a \emph{two-player impartial game}, which we will refer to as \emph{Toggle}, played on graphs other than lattices. Special attention is paid to the generalized Petersen graph $P(m,k)$ (see Definition \ref{def:P(m,k)}).

Following \cite{BCG}, a two-player impartial game refers to a game with the following properties:\\[-7mm]
\begin{enumerate}
    \item Two players alternate moves until a final state is reached, at which point one player is declared the winner. \item At each stage allowable moves depend only on the state of the game and not on which player is moving. 
    \item Both players have perfect information, i.e., both players know the state of the game at all times.
    \item No moves rely on chance. 

\end{enumerate}

The Sprague-Grundy theorem \cite{BK} asserts that all two-player impartial games can be analyzed by assigning a nonnegative integer value, called the \emph{Nimber} (or \emph{Grundy number}), to each position recursively. Under normal play constraints, this theorem implies that either the first player has a winning strategy, denoted as an $N$-game ($N$ for next player), or the second player has a winning strategy, denoted as a $P$-game ($P$ for previous player). Note that the Nimber of a game is $0$ if and only if the game is a second-player win, i.e.\ the second player has a winning strategy regardless of the moves of the first player. Further note that a game is an $N$-game if and only if there exists at least one legal move that results in a $P$-game, whereas it is a $P$-game if and only if there is no legal move or every legal move results in an $N$-game. 

As maintained by the Sprague-Grundy theorem, the Nimber $\mathcal{G}(\mathcal{S})$ of a position $\mathcal{S}$ is determined using the \emph{minimal-excluded rule}: if $T$ is a finite subset of $\mathbb{N} \cup \{0\}$, then 
$$\mex(T)=\min\{(\N\cup\{0\})\setminus T\}.$$
The Nimber of a two-player impartial game with position $\mathcal{S}$ is given by $$\mathcal{G}(\mathcal{S}) = \mex\{\mathcal{G}(\mathcal{S}_1),\mathcal{G}(\mathcal{S}_2),\dots,\mathcal{G}(\mathcal{S}_n)\}$$ where $\mathcal{S}_1,\mathcal{S}_2,\dots,\mathcal{S}_n$ represent all possible positions that occur after one move is played at position $\mathcal{S}$.
Furthermore, if $\mathcal{L}$ represents a position consisting of two independent impartial games with positions $\mathcal{H}$ and $\mathcal{K}$, then $$\mathcal{G}(\mathcal{L}) = \mathcal{G}(\mathcal{H}) \oplus \mathcal{G}(\mathcal{K}),$$
where $x \oplus y$ denotes the bitwise $XOR$ between two nonnegative integers $x$ and $y$.
Later in the paper, we will explain Nimbers in the context of Toggle.

Throughout, let $G$ denote a finite undirected simple graph. The (open) neighborhood of a vertex $v \in V(G)$ is represented by $N(v)$. We denote the closed neighborhood of $v$ by $N[v] = N(v) \cup \{v\}$. Likewise, open and closed sets of vertices at distance at most $r$ from a fixed vertex $v$ are represented by $N_r(v)$ and $N_r[v]$, respectively. For a subset $W \subset V(G)$, we denote the induced subgraph on $W$ by $G[W]$.


\section{Definitions and preliminary results}\label{basic}

The game of Toggle is played on a simple connected graph $G$ where each vertex of $G$ is assigned an initial weight of $0$ or $1$. We denote the weight of a vertex $v$ at stage $j$ by $\omega^{(j)}(v)$, where the initial stage is defined as stage $j=0$. Let $\sigma^{(j)}(v) := \sum \{\omega^{(j)}(u)\!\mid\! {u\in N[v]}\}$ and denote by $V^{(j)}_i(G)$ the set of all vertices of weight $i$ ($i=0,1$) after a $j^{\rm th}$ Toggle move, i.e., $V^{(j)}_i(G) = \{v \in V(G) \mid \omega^{(j)}(v) = i\}$. Finally, we define $\sigma^{(j)}(G) = \left|V_1^{(j)}(G)\right|$.

A \emph{legal Toggle move} at stage $j$ consists of selecting a vertex $v \in V(G)$ with $\omega^{(j)}(v)=1$ and switching the weights of $u$ to $\omega^{(j+1)}(u)=\omega^{(j)}+1 \pmod{2}$ for every $u \in N[v]$ subject to the requirement $\sigma^{(j+1)}(v) < \sigma^{(j)}(v)$. In such case, we refer to the vertex $v \in V(G)$ as \emph{playable}.

Note that the above implies $\sigma^{(j+1)}(G) < \sigma^{(j)}(G)$, and as a consequence a game of Toggle consists of at most $|V(G)|$ moves.





\begin{definition}
We say $v \in V(G)$ is \emph{terminally unplayable at stage $j$} if it becomes unplayable at some stage $k\leq j$ and remains unplayable irrespective of all future moves. We call $v \in V(G)$ \emph{penultimately unplayable} if $v$ becomes terminally unplayable after at most one move at $u \in N[v]$ regardless of the sequence of previous moves. We extend the above terminology to graphs, i.e.\ we call $G$ \emph{terminally unplayable at stage $j$} (resp.\ \emph{penultimately unplayable}) if every vertex $v \in V(G)$ is terminally unplayable at stage $j$ (penultimately unplayable). When the stage is clear from context, we omit $j$ from the notation, simply stating that $v$ or $G$ is terminally (or penultimately) unplayable.
\end{definition}

\begin{remark}\label{rmk:uv}
    It is easy to see that if $\omega^{(j)}(v)=0$ and $u$ is terminally unplayable for all $u \in N(v)$, then $v$ is terminally unplayable as well.
    
\end{remark}
  
\begin{proposition}\label{prop:PathCycle}
    Let $G$ be a finite simple graph given the assignment $V_0^{(0)}(G)=\emptyset$ with $\Delta(G) \le 2$. Then $G$ is penultimately unplayable.
\end{proposition}

\begin{proof}
   Without loss of generality, we may assume $G$ is connected. Suppose first that $G$ is an $n$-path $P_n=v_1v_2\dots v_n$. We proceed by induction on $n$. It is left as an easy exercise to show the base cases $P_1$, $P_2$, $P_3$ with respective assignments $V^{(0)}_0(P_i) = \emptyset$, $i=1,2,3$, are penultimately unplayable.

    Assume $P_m$ with $V^{(0)}_0(P_m) = \emptyset$ is penultimately unplayable for all positive integers $m < n$. Suppose the initial Toggle move on $P_n = v_1 v_2 \dots v_n$ occurs at vertex $v_i$. First suppose $i=1$. (The case $i=n$ is symmetric.) Then $V^{(1)}_0(P_n) = \{v_{1},v_2\}$ and $V^{(1)}_1(P_n) = \{v_3,v_4, \dots v_n\}$. The induction hypothesis implies $P_{n-2} = v_3 v_4 \dots v_{n}$ is penultimately unplayable. So each vertex $v_3,\dots,v_n$ in $P_n$ is penultimately unplayable unless $v_1$ or $v_2$ affects its playability. If $v_1$ and $v_2$ each have weight $0$, then the playability of each vertex $v_3,\dots,v_n$ in $P_n$ is equivalent to the playability of the respective vertices in $P_{n-2}$. Thus to show $P_n$ with $V^{(0)}_0(P_n) = \emptyset$ is penultimately unplayable we need only show $v_{1}$ and $v_2$ are terminally unplayable at stage $1$.

    The only way for $v_2$ to become playable is if $\omega^{(j)}(v_2)=\omega^{(j)}(v_3)=1$ for some $j>1$, and the only way to have $\omega^{(j)}(v_2)=1$ is if a move is made on $v_3$ at stage $k < j$. Thus assume $\omega^{(k-1)}(v_3) = 1$ and the $k^{th}$ Toggle move occurs at $v_3$. Then $\omega^{(k)}(v_2) = 1$ and $\omega^{(k)}(v_1)=\omega^{(k)}(v_3)=\omega^{(k)}(v_4)=0$, so $v_2$ remains unplayable. Since a move was made on $v_3$, $v_4$ is terminally unplayable by the induction hypothesis. So $\omega^{(j)}(v_3)=0$ for all $j>k$, which implies $v_2$ is terminally unplayable. By Remark \ref{rmk:uv}, $v_1$ is also terminally unplayable.

    Next, suppose the initial move is made at $v_i \in V(P_n)$ for $i \ne 1,n$. Then $V^{(1)}_0(P_n) = \{v_{i-1},v_i,v_{i+1}\}$. Define $P_{i} =v_1v_2 \cdots v_{i}$ and $P_{n-i+1}=v_iv_{i+1} \cdots v_n$ with $V^{(1)}_0(P_i) = \{v_{i-1},v_i \}$ and $V^{(1)}_0(P_{n-i+1})= \{v_i,v_{i+1} \}$. Note that after an initial Toggle move at $v_i$ on $P_i$ with $V^{(0)}_0(P_i) = \emptyset$, the resulting assignment is $V^{(1)}_0(P_i) = \{v_{i-1},v_i \}$. Similarly, after an initial Toggle move at $v_i$ on $P_{n-i+1}$, the resultant assignment is $V^{(1)}_0(P_{n-i+1}) = \{v_{i},v_{i+1} \}$. Since $i,n-i+1 <n$, and $V^{(0)}_0(P_i) = \emptyset$ and $V^{(0)}_0(P_{n-i+1}) = \emptyset$, we may conclude by induction that $P_i$ and $P_{n-i+1}$ are penultimately unplayable. Thus each vertex $v_1,\dots, v_{i-2},v_{i+2}, \dots v_n$ in $P_n$ is penultimately unplayable unless $v_{i-1}, v_i, v_{i+1}$ affects its playability.
    
    It remains to show that $v_{i-1}$, $v_i$, and $v_{i+1}$ are terminally unplayable after the initial Toggle move at $v_i$. Suppose the $k^{th}$ Toggle move is made at vertex $v_{i+2}$. Then $\omega^{(k)}(v_{i+1})=1$ and $\omega^{(k)}(v_i)=\omega^{(k)}(v_{i+2})=0$, so $v_{i+1}$ remains unplayable. By the induction hypothesis, $v_{i+3}$ is terminally unplayable once a move has been made on $v_{i+2}$. Thus, $\omega^{(j)}(v_{i+2})=0$ for all $j>k$. It follows that $v_{i+1}$ is terminally unplayable and by symmetry, $v_{i-1}$ is terminally unplayable as well. But then $v_i$ is terminally unplayable by Remark \ref{rmk:uv}. Therefore $P_n$ is penultimately unplayable.

    Now suppose $G$ is a cycle $C_n=v_1v_2\dots v_n$ with initial assignment $V_0^{(0)}=\emptyset$, $n \ge 5$. (The cases $C_3$ and $C_4$ are easily treated and left as an exercise.) Consider the path obtained by removing the edge $v_1v_2$ from $C_n$. By above, the resulting path is penultimately unplayable. It remains to show that $v_1$ and $v_2$ are penultimately unplayable in $C_n$. But this is achieved by removing the edge $v_{n-1}v_n$ since the resulting path is penultimately unplayable and contains $v_1v_2$ as an internal edge.
\end{proof}

\begin{remark}
    Note that the condition $V_0^{(0)}(G) = \emptyset$ in Proposition \ref{prop:PathCycle} is necessary. For example, if $G = P_8$ and $V_0^{(0)}(G) = \{v_6\}$ then consecutive Toggle moves on $v_5$, $v_6$, $v_3$ render $v_5$ playable again, see Figure \ref{fig:PCCounter}.
\end{remark}

\begin{figure}[H]
    \centering
    \begin{tikzpicture}[scale=1.1]
        \coordinate(a) at (-4.5,1);\coordinate(b) at (-3,1);\coordinate(c) at (-1.5,1);\coordinate(d) at (0,1);\coordinate(e) at (1.5,1);\coordinate(f) at (3,1);\coordinate(g) at (4.5,1);\coordinate(h) at (6,1);\foreach\i in{a,b,c,d,e,f,g,h}{\filldraw(\i)circle(0.08);}\draw (a)--(b)--(c)--(d)--(e)--(f)--(g)--(h);
        \node at (-4.5,1.3){$v_{1}$};\node at (-3,1.3){$v_{2}$};\node at (-1.5,1.3) {$v_{3}$};\node at (0,1.3) {$v_{4}$};\node at (1.5,1.3){$v_{5}$};\node at (3,1.3){$v_{6}$};\node at (4.5,1.3) {$v_{7}$};\node at (6,1.3) {$v_{8}$};

        \node at (-4.5,0.5){$1$};\node at (-3,0.5){$1$};\node at (-1.5,0.5) {$1$};\node at (0,0.5) {$1$};\node at (1.5,0.5){$1$};\node at (3,0.5){$0$};\node at (4.5,0.5) {$1$};\node at (6,0.5) {$1$};
        
    \end{tikzpicture}\vspace*{-2mm}
    \caption{An example of an initial state with $\Delta(G) \leq 2$ and $V_{0}^{(0)}(G) \neq \emptyset$.}
    \label{fig:PCCounter}
    
\end{figure}
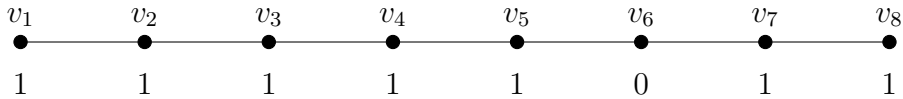

The next result follows immediately from Proposition \ref{prop:PathCycle}. The proof is left as an exercise for the reader.

\begin{corollary}
    Let $G$ be an $n$-path $P_n=v_1v_2\dots v_n$ given the assignment $V_1^{(0)}(P_n)= \{v_{m_1}, v_{m_1+1}, \dots, v_{m_2-1}, v_{m_2} \}$ for some $1 \le m_1 < m_2 \le n$. Then $P_n$ is penultimately unplayable.
\end{corollary}

We are now prepared to prove a result on graphs with maximum degree three.

\begin{proposition}\label{prop:degree3penultimately}
Let $G$ be a simple graph on $n$ vertices with $V_0^{(0)}(G)=\emptyset$ and $\Delta(G) \le 3$. Then $G$ is penultimately unplayable.
\end{proposition}

\begin{proof}
Without loss of generality, we may assume $G$ is connected. We proceed by induction on $n=|V(G)|$, $n \ge 5$. Verification of the base cases, which consist of all graphs on $n \le 4$ vertices, is left to the reader.

    Assume that any simple finite graph $H$ with $m < n$ vertices, $V^{(0)}_0(H)= \emptyset$, and $\Delta(H) = 3$ is penultimately unplayable. Suppose the initial Toggle move occurs at $v \in V(G)$. Define $G_v$ to be the subgraph of $G$ induced on the vertex set $V(G) \setminus N[v]$ and let $\mathcal{C}_1, \mathcal{C}_2, \dots \mathcal{C}_k$ be the components of $G_v$. If $\Delta(\mathcal{C}_i) \le 2$ for $i \in \{1,2,\dots, k\}$, then by Proposition \ref{prop:PathCycle} $\mathcal{C}_i$ is penultimately unplayable. Otherwise, $\Delta(\mathcal{C}_i) = 3$ for some $i \in \{1,2, \dots, k\}$, which implies $\mathcal{C}_i$ is penultimately unplayable by the induction hypothesis. Thus each of these components is penultimately unplayable in $G$ unless some $u \in N[v]$ affects their playability.
    
    It remains to show $u$ is terminally unplayable for all $u \in N[v]$. By Remark \ref{rmk:uv}, we may assume $u \neq v$. Observe that $u$ is clearly terminally unplayable if $\deg(u)=1$, hence assume $\deg(u) >1$. Further observe that $\sigma^{(1)}(u) \leq \deg(u)-1$. Because $\deg(u) \leq 3$, $\deg(u)-1 \leq \left\lceil\frac{\deg(u)}{2}\right\rceil$. So $\sigma^{(1)}(u) \leq \left\lceil\frac{\deg(u)}{2}\right\rceil$, which implies $u$ is unplayable. Let $w \in N(u) \setminus \{v\}$ be the vertex at which the next Toggle move is made. Note that this is only possible if $w \notin N(v)$. Thus $N[w] \setminus \{u\} \subseteq C_i$ for some $i \in \{1,2,\dots,k\}$. We further have $\omega^{(1)}(w)=1$, $\omega^{(2)}(w)=0$, and $\omega^{(2)}(u)=1$, so $\sigma^{(2)}(u) \le \sigma^{(1)}(u)+1-1 = \sigma^{(1)}(u) \leq \deg(u)-1 \leq \left\lceil\frac{\deg(u)}{2}\right\rceil$. Thus $u$ is still unplayable and can only become playable if $\omega^{(j)}(w)=1$ for some $j>2$. Since $w$ has been played, all vertices in $N[w] \setminus \{u\}$ are terminally unplayable by the induction hypothesis. It follows that $\omega^{(j)}(w)=0$ for all $j>2$, thus $u \in N(v)$ is terminally unplayable.
\end{proof}

\begin{remark}
    Note that the condition $V_0^{(0)}(G) = \emptyset$ in Proposition \ref{prop:degree3penultimately} is necessary. For example, if $V_0^{(0)}(G) = \{u_3\}$ then consecutive Toggle moves on $v$, $u_3$, $w_1$, $w_4$ render $v$ playable again, see Figure \ref{fig:Deg3Counter}.
\end{remark}

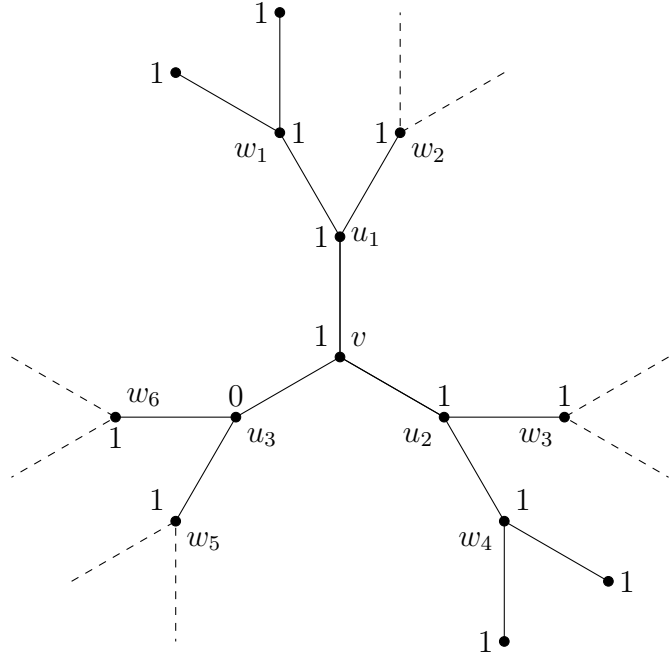
\begin{figure}[H]
    \centering
    \begin{tikzpicture}[scale=0.8]
\coordinate(a1) at (0,0);
\coordinate(b1) at (0,2);
\coordinate(b2) at (1.73,-1);
\coordinate(b3) at (-1.73,-1);
\coordinate(c1) at (-1,3.73);
\coordinate(c2) at (1,3.73);
\coordinate(c3) at (3.73,-1);
\coordinate(c4) at (2.73,-2.73);
\coordinate(c5) at (-2.73,-2.73);
\coordinate(c6) at (-3.73,-1);
\coordinate(d1) at (-2.73,4.73);
\coordinate(d2) at (-1,5.73);
\coordinate(d3) at (2.73,4.73);
\coordinate(d4) at (1,5.73);
\coordinate(d5) at (5.46,0);
\coordinate(d6) at (5.46,-2);
\coordinate(d7) at (4.46,-3.73);
\coordinate(d8) at (2.73,-4.73);
\coordinate(d9) at (-4.46,-3.73);
\coordinate(d10) at (-2.73,-4.73);
\coordinate(d11) at (-5.46,0);
\coordinate(d12) at (-5.46,-2);

\draw(d1)--(c1)--(d2);
\draw[dashed](d3)--(c2)--(d4);
\draw(c1)--(b1)--(c2);
\draw(b1)--(a1)--(b2);
\draw(a1)--(b3);
\draw(c5)--(b3)--(c6);
\draw(b1)--(a1)--(b2);
\draw(c3)--(b2)--(c4);
\draw(d7)--(c4)--(d8);
\draw[dashed](d5)--(c3)--(d6);
\draw[dashed](d9)--(c5)--(d10);
\draw[dashed](d11)--(c6)--(d12);

\foreach\i in{a1,b1,b2,b3,c1,c2,c3,c4,c5,c6,d1,d2,d7,d8}{\filldraw(\i)circle(0.08);}

\node[above right] at (a1){$v$};
\node[right] at (b1){$u_1$};
\node[below left] at (b2){$u_2$};
\node[below right] at (b3){$u_3$};
\node[below left] at (c1){$w_1$};
\node[below right] at (c2){$w_2$};
\node[below left] at (c3){$w_3$};
\node[below left] at (c4){$w_4$};
\node[below right] at (c5){$w_5$};
\node[above right] at (c6){$w_6$};

\node[above left] at (a1){$1$};
\node[left] at (b1){$1$};
\node[above] at (b2){$1$};
\node[above] at (b3){$0$};
\node[right] at (c1){$1$};
\node[left] at (c2){$1$};
\node[above] at (c3){$1$};
\node[above right] at (c4){$1$};
\node[above left] at (c5){$1$};
\node[below] at (c6){$1$};
\node[left] at (d1){$1$};
\node[left] at (d2){$1$};
\node[right] at (d7){$1$};
\node[left] at (d8){$1$};

    \end{tikzpicture}\vspace*{-2mm}
    \caption{An example of an initial state with $\Delta(G) \leq 3$ and $V_{0}^{(0)}(G) = \{u_3\}$. }
    \label{fig:Deg3Counter}
\end{figure}

Since Toggle is an impartial game played on a graph $G$ with a prescribed assignment $V_0^{(j)}$, we can assign a Nimber $\Nim(\mathcal{S})$ to each position $\mathcal{S} = \{G,V_0^{(j)}\}$. (Here we refer to $G$ as the \textit{Toggle graph} of the game.) Observe that any graph $G$ with assignment $V_0^{(j)}=V(G)$ has Nimber zero since the next player has no legal move, i.e.\ it is previous-player winning. We denote all possible positions that occur after a move is played on $\mathcal{S}$ by $\mathcal{S}_1,\mathcal{S}_2,\dots,\mathcal{S}_n$. In this case the Nimber of $\mathcal{S}$ is given by $$\mathcal{G}(\mathcal{S}) = \mex\{\Nim(\mathcal{S}_1),\Nim(\mathcal{S}_2),\dots,\Nim(\mathcal{S}_n)\}.$$
Furthermore, if $G = H + K$ (i.e.\ the disjoint union of graphs $H$ and $K$) and $V_H$ and $V_K$ denote the assignment $V$ on $G$ restricted to $H$ and $K$, then $$\mathcal{G}(\{G,V\}) = \mathcal{G}(\{H,V_H\}) \oplus \mathcal{G}(\{K,V_K\}).$$

\section{The Generalized Petersen Graph \boldmath{$P(m,1)$}}\label{sec:GPm1}

The purpose of this section is to calculate the Nimbers of the game of Toggle played on generalized Petersen graphs $P(m,1)$, $m \ge 3$. This problem quickly reduces to Toggle played on a $2 \times m$ lattice $\mathcal{L}_{2,m}$. We denote by $v_{i,j}$ the vertex in the $i^{th}$ row and $j^{th}$ column of $\mathcal{L}_{2,m}$ where, for future convenience, we assume $0 \leq i \leq 1$ and $1 \leq j \leq m$. Observe that the generalized Petersen graph $P(m,1)$ is equivalent to $\mathcal{L}_{2,m}$ if one adds the edges $v_{0,1}v_{0,m}$ and $v_{1,1}v_{1,m}$ to the latter. See Figure \ref{fig:lattice} which illustrates the case $m = 9$.

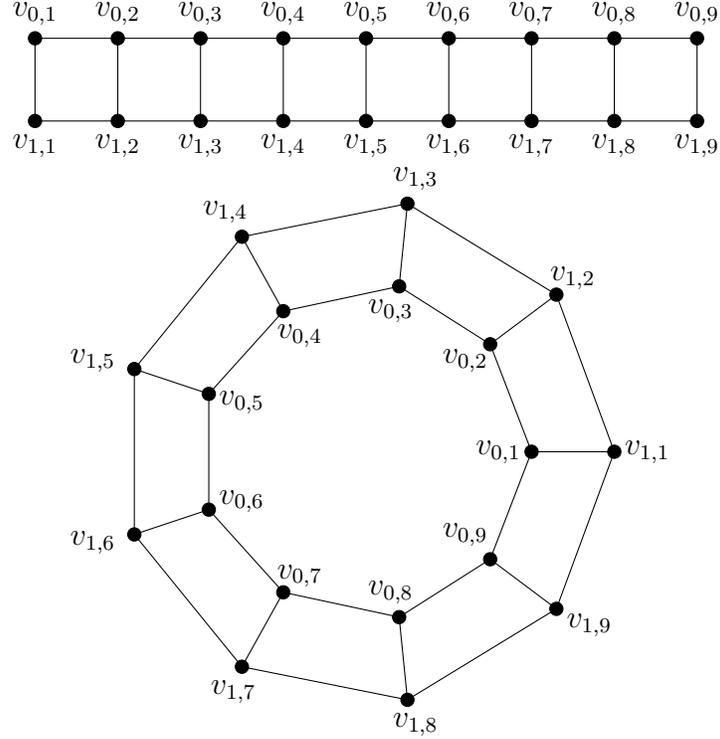
\begin{figure}[H]
    \centering
    \begin{tikzpicture}[scale=1.1]
        \coordinate(A) at (0,0);\coordinate(B) at (1,0);\coordinate(C) at (2,0);\coordinate(D) at (3,0);\coordinate(E) at (4,0);\coordinate(F) at (5,0);\coordinate(G) at (6,0);\coordinate(H) at (7,0);\coordinate(I) at (8,0);\foreach\i in{A,B,C,D,E,F,G,H,I}{\filldraw(\i)circle(0.08);}\draw(A)--(B)--(C)--(D)--(E)--(F)--(G)--(H)--(I);\node at (0,-.3){$v_{1,1}$};\node at (1,-.3){$v_{1,2}$};\node at (2,-.3) {$v_{1,3}$};\node at (3,-.3) {$v_{1,4}$};\node at (4,-.3){$v_{1,5}$};\node at (5,-.3){$v_{1,6}$};\node at (6,-.3) {$v_{1,7}$};\node at (7,-.3) {$v_{1,8}$};\node at (8,-.3) {$v_{1,9}$}; 

        \coordinate(a) at (0,1);\coordinate(b) at (1,1);\coordinate(c) at (2,1);\coordinate(d) at (3,1);\coordinate(e) at (4,1);\coordinate(f) at (5,1);\coordinate(g) at (6,1);\coordinate(h) at (7,1);\coordinate(i) at (8,1);\foreach\i in{a,b,c,d,e,f,g,h,i}{\filldraw(\i)circle(0.08);}\draw (a)--(b)--(c)--(d)--(e)--(f)--(g)--(h)--(i);\draw (A)--(a);\draw (B)--(b);\draw (C)--(c);\draw (D)--(d);\draw (E)--(e);\draw (F)--(f);\draw (G)--(g);\draw (H)--(h);\draw (I)--(i);
        \node at (0,1.3){$v_{0,1}$};\node at (1,1.3){$v_{0,2}$};\node at (2,1.3) {$v_{0,3}$};\node at (3,1.3) {$v_{0,4}$};\node at (4,1.3){$v_{0,5}$};\node at (5,1.3){$v_{0,6}$};\node at (6,1.3) {$v_{0,7}$};\node at (7,1.3) {$v_{0,8}$};\node at (8,1.3) {$v_{0,9}$};

        \coordinate(AA) at (7,-4);
        \coordinate(BB) at (6.3,-2.1);
        \coordinate(CC) at (4.5,-1);
        \coordinate(DD) at (2.5,-1.4);
        \coordinate(EE) at (1.2,-3);
        \coordinate(FF) at (1.2,-5);
        \coordinate(GG) at (2.5,-6.6);
        \coordinate(HH) at (4.5,-7);
        \coordinate(II) at (6.3,-5.9);
        \foreach\i in{AA,BB,CC,DD,EE,FF,GG,HH,II}{\filldraw (\i)circle(0.08);}
        \draw (AA)--(BB)--(CC)--(DD)--(EE)--(FF)--(GG)--(HH)--(II)--(AA);

        \coordinate(aa) at (6,-4);
        \coordinate(bb) at (5.5,-2.7);
        \coordinate(cc) at (4.4,-2);
        \coordinate(dd) at (3,-2.3);
        \coordinate(ee) at (2.1,-3.3);
        \coordinate(ff) at (2.1,-4.7);
        \coordinate(gg) at (3,-5.7);
        \coordinate(hh) at (4.4,-6);
        \coordinate(ii) at (5.5,-5.3);
        \foreach\i in{aa,bb,cc,dd,ee,ff,gg,hh,ii}{\filldraw (\i)circle(0.08);}
        \draw (aa)--(bb)--(cc)--(dd)--(ee)--(ff)--(gg)--(hh)--(ii)--(aa);

        \draw (AA)--(aa);\draw (BB)--(bb);\draw (CC)--(cc);\draw (DD)--(dd);\draw (EE)--(ee);\draw (FF)--(ff);\draw (GG)--(gg);\draw (HH)--(hh);\draw (II)--(ii);

        \node at (7.4,-4){$v_{1,1}$};
        \node at (6.5,-1.9){$v_{1,2}$};
        \node at (4.6,-.7){$v_{1,3}$};
        \node at (2.3,-1.1){$v_{1,4}$};
        \node at (0.7,-2.9){$v_{1,5}$};
        \node at (0.7,-5.1){$v_{1,6}$};
        \node at (2.4,-6.9){$v_{1,7}$};
        \node at (4.6,-7.3){$v_{1,8}$};
        \node at (6.7,-6.1){$v_{1,9}$};

        \node at (5.6,-4){$v_{0,1}$};
        \node at (5.2,-2.9){$v_{0,2}$};
        \node at (4.3,-2.3){$v_{0,3}$};
        \node at (3.2,-2.6){$v_{0,4}$};
        \node at (2.5,-3.4){$v_{0,5}$};
        \node at (2.5,-4.6){$v_{0,6}$};
        \node at (3.2,-5.5){$v_{0,7}$};
        \node at (4.3,-5.7){$v_{0,8}$};
        \node at (5.2,-5){$v_{0,9}$};
        
    \end{tikzpicture}\vspace*{-2mm}
    \caption{A labeling of the vertices of $\mathcal{L}_{2,9}$ (top) and $P(9,1)$.}
    \label{fig:lattice}
\end{figure}

Below we define initial assignments on $\mathcal{L}_{2,m}$ that will be central to what follows. Here we assume $m \geq 3$.
\begin{enumerate}[$(a)$]
    \item\label{item: Hm} For $m \neq 3$, let $\mathcal{H}_m = \{\mathcal{L}_{2,m},V_0^{(0)}\}$ where $V^{(0)}_0 = \{v_{0,1},v_{0,m},v_{1,1},v_{1,2},v_{1,m-1},v_{1,m}\}$, \\see Fig.\ \ref{fig:Hm}. For $m=3$, let $\mathcal{H}_3=\{\mathcal{L}_{2,3},V_0^{(0)}\}$ where $V_0^{(0)} = \{v_{0,1},v_{0,3},v_{1,1},v_{1,3}\}$.
    \item\label{item: Dm} Let $\mathcal{D}_m =\{\mathcal{L}_{2,m},V_0^{(0)}\}$ where $V^{(0)}_0 = \{v_{0,1},v_{0,m-1},v_{0,m},v_{1,1},v_{1,2},v_{1,m}\}$, see Fig.\ \ref{fig:Dm}.
    \item\label{item: Tm} Let $\mathcal{T}_m = \{\mathcal{L}_{2,m}, V^{(0)}_0\}$ where $V^{(0)}_0 = \{v_{0,m-1},v_{0,m},v_{1,m}\}$, see Fig.\ \ref{fig:Tm}.
\end{enumerate}

\begin{figure}[H]
    \centering
    \begin{tikzpicture}[scale=1.2]
        \coordinate(A) at (0,0);\coordinate(B) at (1,0);\coordinate(C) at (2,0);\coordinate(D) at (3,0);\coordinate(E) at (4,0);\coordinate(F) at (5,0);\coordinate(G) at (6,0);\coordinate(H) at (7,0);\coordinate(I) at (8,0);\foreach\i in{A,B,C,D,E,F,G,H,I}{\filldraw(\i)circle(0.08);}\draw (A)--(B)--(C)--(D)--(E)--(F)--(G)--(H)--(I);\node at (0,-.3){$0$};\node at (1,-.3){$0$};\node at (2,-.3) {$1$};\node at (3,-.3) {$1$};\node at (4,-.3){$1$};\node at (5,-.3){$1$};\node at (6,-.3) {$1$};\node at (7,-.3) {$0$};\node at (8,-.3) {$0$}; 

        \coordinate(a) at (0,1);\coordinate(b) at (1,1);\coordinate(c) at (2,1);\coordinate(d) at (3,1);\coordinate(e) at (4,1);\coordinate(f) at (5,1);\coordinate(g) at (6,1);\coordinate(h) at (7,1);\coordinate(i) at (8,1);\foreach\i in{a,b,c,d,e,f,g,h,i}{\filldraw(\i)circle(0.08);}\draw (a)--(b)--(c)--(d)--(e)--(f)--(g)--(h)--(i);\draw(A)--(a);\draw(B)--(b);\draw(C)--(c);\draw(D)--(d);\draw(E)--(e);\draw(F)--(f);\draw(G)--(g);\draw(H)--(h);\draw(I)--(i);
        \node at (0,1.3){$0$};\node at (1,1.3){$1$};\node at (2,1.3) {$1$};\node at (3,1.3) {$1$};\node at (4,1.3){$1$};\node at (5,1.3){$1$};\node at (6,1.3) {$1$};\node at (7,1.3) {$1$};\node at (8,1.3) {$0$};
    \end{tikzpicture}\vspace*{-2mm}
    \caption{$\mathcal{H}_9$ with $V^{(0)}_0=\{v_{0,1},v_{0,9},v_{1,1},v_{1,2},v_{1,8},v_{1,9}\}$.}
    \label{fig:Hm}
\end{figure}
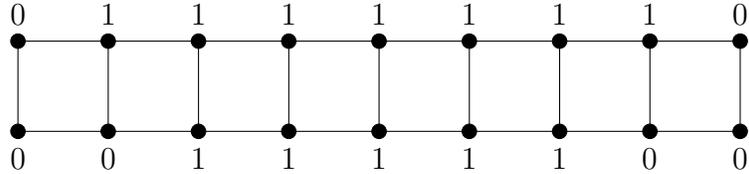

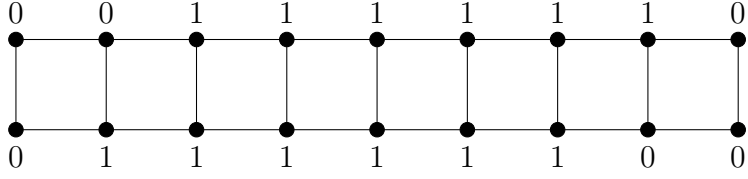
\begin{figure}[H]
    \centering
    \begin{tikzpicture}[scale=1.2]
        \coordinate(A) at (0,0);\coordinate(B) at (1,0);\coordinate(C) at (2,0);\coordinate(D) at (3,0);\coordinate(E) at (4,0);\coordinate(F) at (5,0);\coordinate(G) at (6,0);\coordinate(H) at (7,0);\coordinate(I) at (8,0);\foreach\i in{A,B,C,D,E,F,G,H,I}{\filldraw(\i)circle(0.08);}\draw (A)--(B)--(C)--(D)--(E)--(F)--(G)--(H)--(I);\node at (0,-.3){$0$};\node at (1,-.3){$1$};\node at (2,-.3) {$1$};\node at (3,-.3) {$1$};\node at (4,-.3){$1$};\node at (5,-.3){$1$};\node at (6,-.3) {$1$};\node at (7,-.3) {$0$};\node at (8,-.3) {$0$}; 

        \coordinate(a) at (0,1);\coordinate(b) at (1,1);\coordinate(c) at (2,1);\coordinate(d) at (3,1);\coordinate(e) at (4,1);\coordinate(f) at (5,1);\coordinate(g) at (6,1);\coordinate(h) at (7,1);\coordinate(i) at (8,1);\foreach\i in{a,b,c,d,e,f,g,h,i}{\filldraw(\i)circle(0.08);}\draw(a)--(b)--(c)--(d)--(e)--(f)--(g)--(h)--(i);\draw(A)--(a);\draw(B)--(b);\draw(C)--(c);\draw(D)--(d);\draw(E)--(e);\draw(F)--(f);\draw(G)--(g);\draw(H)--(h);\draw(I)--(i);
        \node at (0,1.3){$0$};\node at (1,1.3){$0$};\node at (2,1.3) {$1$};\node at (3,1.3) {$1$};\node at (4,1.3){$1$};\node at (5,1.3){$1$};\node at (6,1.3) {$1$};\node at (7,1.3) {$1$};\node at (8,1.3) {$0$};
    \end{tikzpicture}\vspace*{-2mm}
    \caption{$\mathcal{D}_9$ with $V^{(0)}_0=\{v_{0,1},v_{0,2},v_{0,9},v_{1,1},v_{1,8},v_{1,9}\}$.}
    \label{fig:Dm}
\end{figure}

\begin{figure}[H]
    \centering
    \begin{tikzpicture}[scale=1.2]
        \coordinate(A) at (0,0);\coordinate(B) at (1,0);\coordinate(C) at (2,0);\coordinate(D) at (3,0);\coordinate(E) at (4,0);\coordinate(F) at (5,0);\coordinate(G) at (6,0);\coordinate(H) at (7,0);\coordinate(I) at (8,0);\foreach\i in{A,B,C,D,E,F,G,H,I}{\filldraw(\i)circle(0.08);}\draw(A)--(B)--(C)--(D)--(E)--(F)--(G)--(H)--(I);\node at (0,-.3){$1$};\node at (1,-.3){$1$};\node at (2,-.3) {$1$};\node at (3,-.3) {$1$};\node at (4,-.3){$1$};\node at (5,-.3){$1$};\node at (6,-.3) {$1$};\node at (7,-.3) {$1$};\node at (8,-.3) {$0$}; 

        \coordinate(a) at (0,1);\coordinate(b) at (1,1);\coordinate(c) at (2,1);\coordinate(d) at (3,1);\coordinate(e) at (4,1);\coordinate(f) at (5,1);\coordinate(g) at (6,1);\coordinate(h) at (7,1);\coordinate(i) at (8,1);\foreach\i in{a,b,c,d,e,f,g,h,i}{\filldraw(\i)circle(0.08);}\draw(a)--(b)--(c)--(d)--(e)--(f)--(g)--(h)--(i);\draw(A)--(a);\draw[](B)--(b);\draw(C)--(c);\draw(D)--(d);\draw(E)--(e);\draw(F)--(f);\draw(G)--(g);\draw(H)--(h);\draw(I)--(i);
        \node at (0,1.3){$1$};\node at (1,1.3){$1$};\node at (2,1.3) {$1$};\node at (3,1.3) {$1$};\node at (4,1.3){$1$};\node at (5,1.3){$1$};\node at (6,1.3) {$1$};\node at (7,1.3) {$0$};\node at (8,1.3) {$0$};
    \end{tikzpicture}\vspace*{-2mm}
    \caption{$\mathcal{T}_9$ with $V^{(0)}_0=\{v_{0,m-1},v_{0,m},v_{1,m}\}$.}
    \label{fig:Tm}
\end{figure}
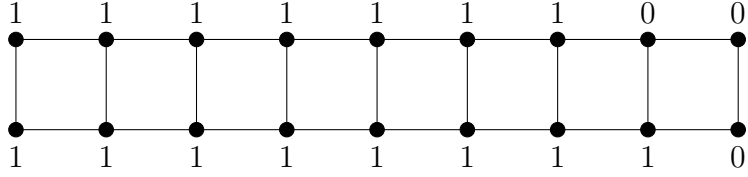

We first consider the Nimbers of Toggle played on $P(m,1)$, $m \geq 3$, with initial assignment $V_0^{(0)}=\emptyset$.

\begin{proposition}\label{prop:GP(m,1)H}
    \textcolor{black} For $V_0^{(0)}=\emptyset$, {$\Nim(\{P(m,1),V_0^{(0)}\}) = \mex\{x_1,x_2,\dots,x_{2m}\}$ where $x_i = \Nim(\mathcal{H}_{m+1})$ for all $i \in \{1,2,\dots,2m\}$.}
\end{proposition}

\begin{proof}
    With initial assignment $V_0^{(0)}=\emptyset$, all resulting assignments after the initial move on $P(m,1)$ are equivalent by symmetry, so without loss of generality assume the initial move is made at $v_{0,1}$. By Proposition \ref{prop:degree3penultimately}, $v_{0,1}$, and $v_{1,1}$ are terminally unplayable. We subdivide the edges $v_{0,1}v_{0,m}$ and $v_{1,1}v_{1,m}$ where the new internal vertices,   $v_{0,m+1}$ and $v_{1,m+1}$, are adjacent and have weight $0$. Note that this does not affect the Toggle game, and that $v_{0,m+1}$ and $v_{1,m+1}$ are terminally unplayable since all weights remain unchanged when reflecting the graph in such a manner that $v_{0,m+1}$ and $v_{1,m+1}$ become $v_{0,1}$ and $v_{1,1}$, respectively. We next remove the edges $v_{0,1}v_{0,m+1}$ and $v_{1,1}v_{1,m+1}$, resulting in the lattice $\mathcal{H}_{m+1}$. Thus, $P(m,1)$ with initial assignment $V_0^{(0)}=\emptyset$ reduces to $\mathcal{H}_{m+1}$ after any initial move. It therefore follows that $\Nim(\{P(m,1),\emptyset\}) = \mex\{x_1,x_2,\dots,x_{2m}\}$ where $x_i = \Nim(\mathcal{H}_{m+1})$ as claimed.
\end{proof}

\begin{remark}
    Note that since $x_i= \Nim(\mathcal{H}_{m+1})$ for all $i$,  Proposition \ref{prop:GP(m,1)H} is equivalent to $\Nim(\{P(m,1),\emptyset\})=0$ if $\Nim(\mathcal{H}_{m+1}) \ge 1$ and $\Nim(\{P(m,1),\emptyset\})=1$ if $\Nim(\mathcal{H}_{m+1}) = 0$.
\end{remark}

\begin{proposition}\label{prop:nimH}
    $\Nim(\mathcal{H}_m) = \mex\{x_3,x_4,\dots,x_s,y_3,y_4,\dots,y_s\}$ where $x_i=\Nim(\mathcal{H}_i) \oplus \Nim(\mathcal{H}_{m+1-i})$ and $y_i=\Nim(\mathcal{D}_i)\oplus \Nim(\mathcal{D}_{m+1-i})$ for $i \in \{3,4,\dots,s\}$ and $s=\lfloor\frac{m+1}{2}\rfloor$.
    
\end{proposition}

\begin{proof}
    First suppose the initial move on $\mathcal{H}_m$ is made at $v_{1,i}$ for some $i \in \{3,\dots,m-2\}$. This results in the two assignments $\mathcal{H}_i$ and $\mathcal{H}_{m+1-i}$ where $V(\mathcal{H}_i)=\{v_{0,1},\dots,v_{0,i},v_{1,1},\dots,v_{1,i}\}$ and $V(\mathcal{H}_{m+1-i})=\{v_{0,i},\dots,v_{0,m},v_{1,i}\dots,v_{1,m}\}$. Note that $V(\mathcal{H}_i) \cap V(\mathcal{H}_{m+1-i}) = \{v_{0,i},v_{1,i}\}$. Note that this argument also applies when the initial move is made at $v_{0,i}$. However, in this case we obtain $\mathcal{D}_i$ and $\mathcal{D}_{m+1-i}$ as the resulting components. By Proposition \ref{prop:degree3penultimately}, both $v_{0,i}$ and $v_{1,i}$ are terminally unplayable, hence a Toggle move at $v_{1,i}$ (resp.\ $v_{0,i}$) results in an assignment that is equivalent to that of a graph with components $\mathcal{H}_i$ and $\mathcal{H}_{m+1-i}$ (resp.\ $\mathcal{D}_i$ and $\mathcal{D}_{m+1-i}$). This gives all possible assignments reachable after a single Toggle move on $\mathcal{H}_m$, hence $\Nim(\mathcal{H}_m) = \mex\{x_3,x_4,\dots,x_s,y_3,y_4,\dots,y_s\}$ where $x_i=\Nim(\mathcal{H}_i) \oplus \Nim(\mathcal{H}_{m+1-i})$ and $y_i=\Nim(\mathcal{D}_i)\oplus \Nim(\mathcal{D}_{m+1-i})$ for $i \in \{3,4,\dots,s\}$ and $s=\lfloor\frac{m+1}{2}\rfloor$.
\end{proof}

\begin{proposition}\label{prop:nimD}
    $\Nim(\mathcal{D}_m) = \mex\{x_3,x_4,\dots,x_{m-2}\}$ where $x_i=\Nim(\mathcal{H}_i)\oplus\Nim(\mathcal{D}_{m+1-i})$ for $i \in \{3,4,\dots,m-2\}$.
\end{proposition}

\begin{proof}
    Suppose the initial Toggle move on $\mathcal{D}_m$ is made at $v_{0,i}$. This results in two assignments $\mathcal{H}_i$ and $\mathcal{D}_{m+1-i}$, where $V(\mathcal{H}_i)=\{v_{0,1},\dots,v_{0,i},v_{1,1},\dots,v_{1,i}\}$ and $V(\mathcal{D}_{m+1-i})=\{v_{0,i},\dots,v_{0,m},v_{1,i},\dots,v_{1,m}\}$. By Proposition \ref{prop:degree3penultimately}, $v_{0,i}$ and $v_{1,i}$ are terminally unplayable, hence a Toggle move on $v_{0,i}$ results in an assignment that is equivalent to that of a graph with components $\mathcal{H}_i$ and $\mathcal{D}_{m+1-i}$. This gives all possible assignments reachable after a single Toggle move on $\mathcal{D}_m$, hence $\Nim(\mathcal{D}_m) = \mex\{x_3,x_4,\dots,x_{m-2}\}$ where $x_i=\Nim(\mathcal{H}_i)\oplus\Nim(\mathcal{D}_{m+1-i})$ for $i \in \{3,4,\dots,m-2\}$.
\end{proof}

Propositions \ref{prop:GP(m,1)H}, \ref{prop:nimH}, and \ref{prop:nimD} allow one to recursively calculate the Nimber $\mathcal{G}(P(m,1))$. The code that computes these Nimbers is included in Appendix A.

\section{The Generalized Petersen Graph \boldmath{$P(m,k)$}}\label{sec:GPmk}

\begin{definition}\label{def:P(m,k)}
   Let $P(m,k)$ ($2 \le k < m$) be the generalized Petersen graph, that is, a connected $3$-regular graph consisting of an outer $m$-cycle $\{m,1\}$ and an inner star polygon $\{m,k\}$ with edges adjoining corresponding vertices in the inner and outer graphs. Here the notation $\{m,k\}$ reflects the fact that the inner star polygon is the distance-$k$ graph of an $m$-cycle which may or may not be connected.
   (Note that we may assume $k \leq \floor{\frac{m}{2}}$ because $P(m, k) \cong P(m, m-k)$.)
\end{definition}

We denote the vertices of $\{m,k\}$ by $V_{\{m,k\}} = \{v_{0,1},v_{0,2},\dots, v_{0,m}\}$ and those of $\{m,1\}$ by $V_{\{m,1\}} = \{v_{1,1},v_{1,2},\dots, v_{1,m}\}$. See Fig.\ \ref{fig:GP92} where the case $m=9$, $k=2$ is depicted.
We introduce the following notation for initial assignments on $P(m,k)$:
\vspace{-2mm}
 \begin{enumerate}
    \item\label{item: GP01} $\mathcal{P}_{0,1}(m,k)$ refers to $P(m,k)$ where $V^{(0)}_0\big(P(m,k)\big) = V_{\{m,k\}}$ and $V^{(0)}_1\big(P(m,k)\big) = V_{\{m,1\}}$.
     \item\label{item: GP10} $\mathcal{P}_{1,0}(m,k)$ refers to $P(m,k)$ where $V^{(0)}_1\big(P(m,k)\big) = V_{\{m,k\}}$ and 
$V^{(0)}_0\big(P(m,k)\big) = V_{\{m,1\}}$.
     \item\label{item: GP11} $\mathcal{P}_{1,1}(m,k)$ refers to $P(m,k)$ where $V^{(0)}_0\big(P(m,k)\big) = \emptyset$.
\end{enumerate}

See Figures \ref{fig:GP01}, \ref{fig:GP10}, \ref{fig:GP11} where the vertex weights for the cases 1, 2, 3 with $m=9$ and $k=2$ are indicated.

\begin{figure}[H]
    \centering
    \begin{tikzpicture}[scale=1.1]
        \coordinate(AA) at (7,-4);
        \coordinate(BB) at (6.3,-2.1);
        \coordinate(CC) at (4.5,-1);
        \coordinate(DD) at (2.5,-1.4);
        \coordinate(EE) at (1.2,-3);
        \coordinate(FF) at (1.2,-5);
        \coordinate(GG) at (2.5,-6.6);
        \coordinate(HH) at (4.5,-7);
        \coordinate(II) at (6.3,-5.9);
        \foreach\i in{AA,BB,CC,DD,EE,FF,GG,HH,II}{\filldraw(\i)circle(0.08);}
        \draw(AA)--(BB)--(CC)--(DD)--(EE)--(FF)--(GG)--(HH)--(II)--(AA);

        \coordinate(aa) at (6,-4);
        \coordinate(bb) at (5.5,-2.7);
        \coordinate(cc) at (4.4,-2);
        \coordinate(dd) at (3,-2.3);
        \coordinate(ee) at (2.1,-3.3);
        \coordinate(ff) at (2.1,-4.7);
        \coordinate(gg) at (3,-5.7);
        \coordinate(hh) at (4.4,-6);
        \coordinate(ii) at (5.5,-5.3);
        \foreach\i in{aa,bb,cc,dd,ee,ff,gg,hh,ii}{\filldraw(\i)circle(0.08);}
        \draw(aa)--(cc)--(ee)--(gg)--(ii)--(bb)--(dd)--(ff)--(hh)--(aa);

        \draw(AA)--(aa);\draw(BB)--(bb);\draw(CC)--(cc);\draw(DD)--(dd);\draw(EE)--(ee);\draw(FF)--(ff);\draw(GG)--(gg);\draw(HH)--(hh);\draw(II)--(ii);

        \node at (7.4,-4){$v_{1,1}$};
        \node at (6.5,-1.9){$v_{1,2}$};
        \node at (4.6,-.7){$v_{1,3}$};
        \node at (2.3,-1.1){$v_{1,4}$};
        \node at (0.7,-2.9){$v_{1,5}$};
        \node at (0.7,-5.1){$v_{1,6}$};
        \node at (2.4,-6.9){$v_{1,7}$};
        \node at (4.6,-7.3){$v_{1,8}$};
        \node at (6.7,-6.1){$v_{1,9}$};

        \node at (6.2,-3.8){$v_{0,1}$};
        \node at (5.9,-2.9){$v_{0,2}$};
        \node at (4.8,-2.1){$v_{0,3}$};
        \node at (2.6,-2.4){$v_{0,4}$};
        \node at (1.8,-3.6){$v_{0,5}$};
        \node at (2,-5){$v_{0,6}$};
        \node at (2.6,-5.7){$v_{0,7}$};
        \node at (4,-6.2){$v_{0,8}$};
        \node at (5.4,-5.6){$v_{0,9}$};
        
    \end{tikzpicture}\vspace*{-2mm}
    \caption{A labeling of the vertices of $P(9,2)$.}
    \label{fig:GP92}
\end{figure}
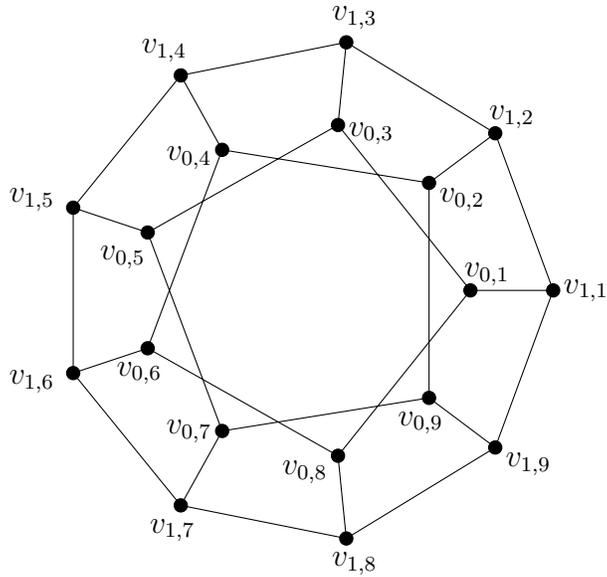

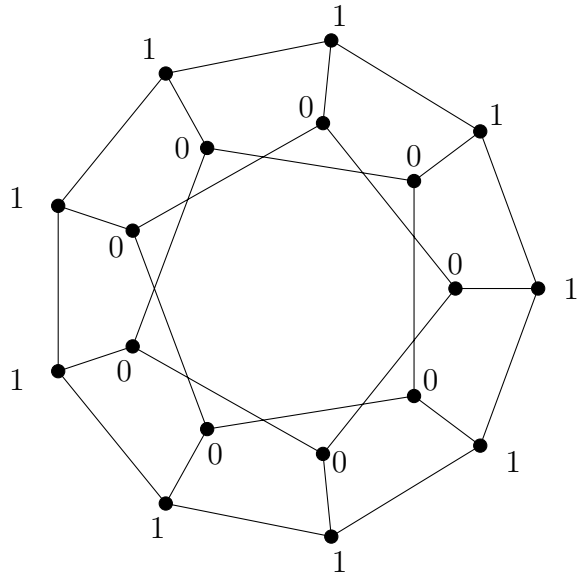
\begin{figure}[H]
    \centering
    \begin{tikzpicture}[scale=1.1]
        \coordinate(AA) at (7,-4);
        \coordinate(BB) at (6.3,-2.1);
        \coordinate(CC) at (4.5,-1);
        \coordinate(DD) at (2.5,-1.4);
        \coordinate(EE) at (1.2,-3);
        \coordinate(FF) at (1.2,-5);
        \coordinate(GG) at (2.5,-6.6);
        \coordinate(HH) at (4.5,-7);
        \coordinate(II) at (6.3,-5.9);
        \foreach\i in{AA,BB,CC,DD,EE,FF,GG,HH,II}{\filldraw(\i)circle(0.08);}
        \draw(AA)--(BB)--(CC)--(DD)--(EE)--(FF)--(GG)--(HH)--(II)--(AA);

        \coordinate(aa) at (6,-4);
        \coordinate(bb) at (5.5,-2.7);
        \coordinate(cc) at (4.4,-2);
        \coordinate(dd) at (3,-2.3);
        \coordinate(ee) at (2.1,-3.3);
        \coordinate(ff) at (2.1,-4.7);
        \coordinate(gg) at (3,-5.7);
        \coordinate(hh) at (4.4,-6);
        \coordinate(ii) at (5.5,-5.3);
        \foreach\i in{aa,bb,cc,dd,ee,ff,gg,hh,ii}{\filldraw(\i)circle(0.08);}
        \draw(aa)--(cc)--(ee)--(gg)--(ii)--(bb)--(dd)--(ff)--(hh)--(aa);

        \draw(AA)--(aa);\draw(BB)--(bb);\draw(CC)--(cc);\draw(DD)--(dd);\draw(EE)--(ee);\draw(FF)--(ff);\draw(GG)--(gg);\draw(HH)--(hh);\draw(II)--(ii);

        \node at (7.4,-4){$1$};
        \node at (6.5,-1.9){$1$};
        \node at (4.6,-.7){$1$};
        \node at (2.3,-1.1){$1$};
        \node at (0.7,-2.9){$1$};
        \node at (0.7,-5.1){$1$};
        \node at (2.4,-6.9){$1$};
        \node at (4.6,-7.3){$1$};
        \node at (6.7,-6.1){$1$};

        \node at (6,-3.7){$0$};
        \node at (5.5,-2.4){$0$};
        \node at (4.2,-1.8){$0$};
        \node at (2.7,-2.3){$0$};
        \node at (1.9,-3.5){$0$};
        \node at (2,-5){$0$};
        \node at (3.1,-6){$0$};
        \node at (4.6,-6.1){$0$};
        \node at (5.7,-5.1){$0$};
        
    \end{tikzpicture}\vspace*{-2mm}
    \caption{$\mathcal{P}_{0,1}(9,2)$ where $\omega^{(0)}(v_{1,i})=1$ and $\omega^{(0)}(v_{0,i})=0$, $i=\{1,2,\dots,9\}$.}
    \label{fig:GP01}
\end{figure}

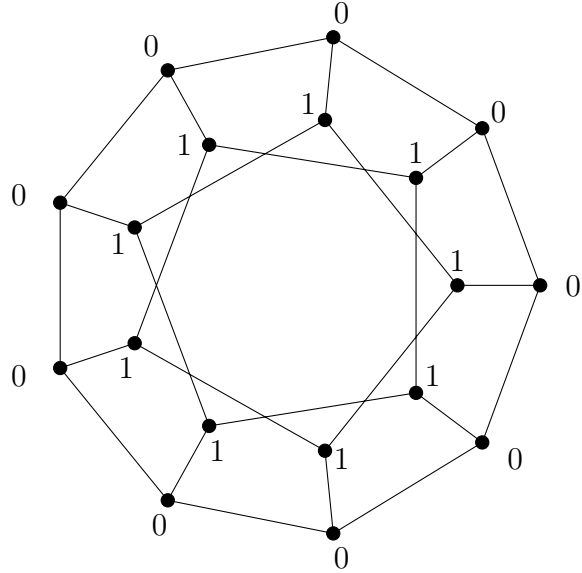
\begin{figure}[H]
    \centering
    \begin{tikzpicture}[scale=1.1]
        \coordinate(AA) at (7,-4);
        \coordinate(BB) at (6.3,-2.1);
        \coordinate(CC) at (4.5,-1);
        \coordinate(DD) at (2.5,-1.4);
        \coordinate(EE) at (1.2,-3);
        \coordinate(FF) at (1.2,-5);
        \coordinate(GG) at (2.5,-6.6);
        \coordinate(HH) at (4.5,-7);
        \coordinate(II) at (6.3,-5.9);
        \foreach\i in{AA,BB,CC,DD,EE,FF,GG,HH,II}{\filldraw(\i)circle(0.08);}
        \draw(AA)--(BB)--(CC)--(DD)--(EE)--(FF)--(GG)--(HH)--(II)--(AA);

        \coordinate(aa) at (6,-4);
        \coordinate(bb) at (5.5,-2.7);
        \coordinate(cc) at (4.4,-2);
        \coordinate(dd) at (3,-2.3);
        \coordinate(ee) at (2.1,-3.3);
        \coordinate(ff) at (2.1,-4.7);
        \coordinate(gg) at (3,-5.7);
        \coordinate(hh) at (4.4,-6);
        \coordinate(ii) at (5.5,-5.3);
        \foreach\i in{aa,bb,cc,dd,ee,ff,gg,hh,ii}{\filldraw(\i)circle(0.08);}
        \draw(aa)--(cc)--(ee)--(gg)--(ii)--(bb)--(dd)--(ff)--(hh)--(aa);

        \draw(AA)--(aa);\draw(BB)--(bb);\draw(CC)--(cc);\draw(DD)--(dd);\draw(EE)--(ee);\draw(FF)--(ff);\draw(GG)--(gg);\draw(HH)--(hh);\draw(II)--(ii);

        \node at (7.4,-4){$0$};
        \node at (6.5,-1.9){$0$};
        \node at (4.6,-.7){$0$};
        \node at (2.3,-1.1){$0$};
        \node at (0.7,-2.9){$0$};
        \node at (0.7,-5.1){$0$};
        \node at (2.4,-6.9){$0$};
        \node at (4.6,-7.3){$0$};
        \node at (6.7,-6.1){$0$};

        \node at (6,-3.7){$1$};
        \node at (5.5,-2.4){$1$};
        \node at (4.2,-1.8){$1$};
        \node at (2.7,-2.3){$1$};
        \node at (1.9,-3.5){$1$};
        \node at (2,-5){$1$};
        \node at (3.1,-6){$1$};
        \node at (4.6,-6.1){$1$};
        \node at (5.7,-5.1){$1$};
        
    \end{tikzpicture}\vspace*{-2mm}
    \caption{$\mathcal{P}_{1,0}(9,2)$ where $\omega^{(0)}(v_{1,i})=0$ and $\omega^{(0)}(v_{0,i})=1$, $i=\{1,2,\dots,9\}$.}
    \label{fig:GP10}
\end{figure}

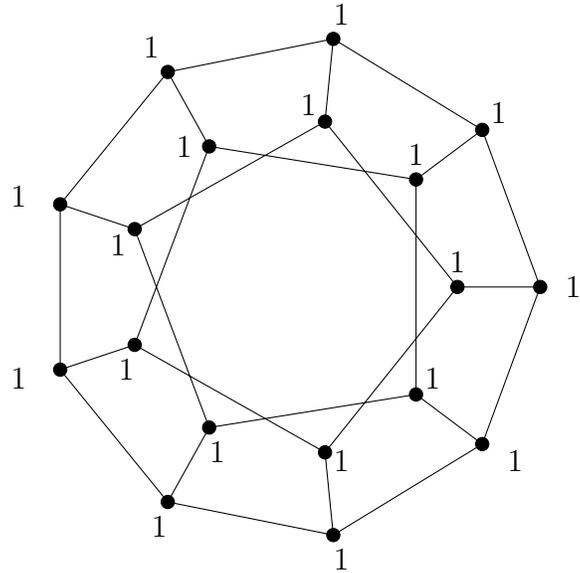
\begin{figure}[H]
    \centering
    \begin{tikzpicture}[scale=1.1]
        \coordinate(AA) at (7,-4);
        \coordinate(BB) at (6.3,-2.1);
        \coordinate(CC) at (4.5,-1);
        \coordinate(DD) at (2.5,-1.4);
        \coordinate(EE) at (1.2,-3);
        \coordinate(FF) at (1.2,-5);
        \coordinate(GG) at (2.5,-6.6);
        \coordinate(HH) at (4.5,-7);
        \coordinate(II) at (6.3,-5.9);
        \foreach\i in{AA,BB,CC,DD,EE,FF,GG,HH,II}{\filldraw(\i)circle(0.08);}
        \draw(AA)--(BB)--(CC)--(DD)--(EE)--(FF)--(GG)--(HH)--(II)--(AA);

        \coordinate(aa) at (6,-4);
        \coordinate(bb) at (5.5,-2.7);
        \coordinate(cc) at (4.4,-2);
        \coordinate(dd) at (3,-2.3);
        \coordinate(ee) at (2.1,-3.3);
        \coordinate(ff) at (2.1,-4.7);
        \coordinate(gg) at (3,-5.7);
        \coordinate(hh) at (4.4,-6);
        \coordinate(ii) at (5.5,-5.3);
        \foreach\i in{aa,bb,cc,dd,ee,ff,gg,hh,ii}{\filldraw(\i)circle(0.08);}
        \draw(aa)--(cc)--(ee)--(gg)--(ii)--(bb)--(dd)--(ff)--(hh)--(aa);

        \draw(AA)--(aa);\draw(BB)--(bb);\draw(CC)--(cc);\draw(DD)--(dd);\draw(EE)--(ee);\draw(FF)--(ff);\draw(GG)--(gg);\draw(HH)--(hh);\draw(II)--(ii);

        \node at (7.4,-4){$1$};
        \node at (6.5,-1.9){$1$};
        \node at (4.6,-.7){$1$};
        \node at (2.3,-1.1){$1$};
        \node at (0.7,-2.9){$1$};
        \node at (0.7,-5.1){$1$};
        \node at (2.4,-6.9){$1$};
        \node at (4.6,-7.3){$1$};
        \node at (6.7,-6.1){$1$};

        \node at (6,-3.7){$1$};
        \node at (5.5,-2.4){$1$};
        \node at (4.2,-1.8){$1$};
        \node at (2.7,-2.3){$1$};
        \node at (1.9,-3.5){$1$};
        \node at (2,-5){$1$};
        \node at (3.1,-6){$1$};
        \node at (4.6,-6.1){$1$};
        \node at (5.7,-5.1){$1$};
        
    \end{tikzpicture}\vspace*{-2mm}
    \caption{$\mathcal{P}_{1,1}(9,2)$ where $\omega^{(0)}(v_{1,i})=\omega^{(0)}(v_{0,i})=1$, $i=\{1,2,\dots,9\}$.}
    \label{fig:GP11}
\end{figure}

The following result is due to Steimle and Staton in \cite{SS09}.

\begin{theorem}\label{thm:Steimle}\leavevmode
    \begin{enumerate}
        \item Let $m \ge 5$ and $\gcd(m,k)=\gcd(m,\ell)=1$ for $2 \le k,\ell \le m-2$. If $P(m,k) \cong P(m,\ell)$, then either $\ell \mod{\pm k} m$ or $k\ell \mod{\pm 1} m$.
        \item Let $m > 3$ and $k,\ell$ relatively prime to $m$ with $k\ell \mod{1} m$. Then $P(m,k) \cong P(m,\ell)$.
    \end{enumerate}
\end{theorem}

\begin{corollary}\label{cor:isomorphism}
    Let $1 \leq k_1,k_2\leq \floor{\frac{m}{2}}$ with $k_1k_2 \equiv 1 \pmod m$. Then we have the following:
    \begin{enumerate}
        \item[\rm{(1)}] $\mathcal{G}(\{P(m,k_1),V\}) = \mathcal{G}(\{P(m,k_2),V\})$ for any assignment $V$ on $P(m,k_1)$.
        \item[\rm{(2)}] $\mathcal{G}(\mathcal{P}_{0,1}(m,k_1)) = \mathcal{G}(\mathcal{P}_{1,0}(m,k_2))$.
    \end{enumerate}
    
\end{corollary}\begin{proof}
(1) follows since $P(m,k_1)$ and $P(m,k_2)$ are isomorphic graphs. (2) is a consequence of the isomorphism $\varphi$ defined by Steimle and Staton that maps the set of inner vertices of $P(m,k_1)$ to the set of outer vertices of $P(m,k_2)$, see \cite[Theorem 1]{SS09}.
\end{proof}

\begin{remark}
    Observe that if $\gcd(m,k)=1$ and there does not exist $\ell \le \lfloor\frac{m-1}{2}\rfloor$ with $\ell \ne k$ such that $k\ell \mod{1} m$, then $\mathcal{G}(\mathcal{P}_{1,0}(m,k)) = \mathcal{G}(\mathcal{P}_{0,1}(m,k))$. Indeed, if $k^2 \not\mod{ 1} m$, then $k\ell \mod{ 1} m$ with $\ell \ne k$. The result follows from Corollary \ref{cor:isomorphism}(2) with $k_1=k_2$.
\end{remark}

\begin{theorem}
    Let $k \in \mathbb{N}$ be even. Then $\mathcal{G}(\mathcal{P}_{1,0}(3k,k))=0$.
\end{theorem}

\begin{proof}
    First note that the graph $P(3k,k)$ consists of an outer $3k$-polygon and $k$ inner triangles, so let us label the vertex sets of each of these triangles as $T_1,T_2,\dots,T_k$, where $T_i = \{v_{0,i},v_{0,i+k},v_{0,i+2k}\}$ for each $i \in \{1,2,\dots, k\}$. We define $O_i = \{v_{1,i},v_{1,i+k},v_{1,i+2k}\}$. Note that a Toggle move on the inner vertex $v_{0,i+\varepsilon k}$ of $T_i$ ($\varepsilon \in \{0,1,2\}$) will switch all vertices in $T_i$ to weight $0$ and $v_{1,i+\varepsilon k}$ to weight $1$. All other vertices remain unchanged.
    
    Suppose a game of Toggle is played with initial position $\mathcal{P}_{1,0}(3k,k)$. Then the second player has a winning strategy provided they can ensure that no outer vertex becomes playable over the entire game. Since a move at a vertex of $T_i$ renders all vertices of $T_i$ unplayable, and since there are $k$ such moves, where $k$ is even, this results in a second player winning game. Conversely, the only way the first player can possibly have a winning strategy is if they can ensure at some stage of the game that at least one outer vertex becomes playable.
    
    We say an outer vertex $v_{1,j}$ is \textit{blocked} at stage $\ell$ if $\omega^{(\ell)}(v_{1,j})=\omega^{(\ell)}(v_{0,j})=0$. Note that if $v_{1,j}$ is blocked then, provided all other outer vertices are unplayable, one can never have $\omega^{(\ell+1)}(v_{1,j})=1$. Player $1$ endeavors to reach a stage where three consecutive outer vertices have weight $1$. However, Player $2$ can always prevent this by blocking the appropriate outer vertices. For example, assume without loss of generality that Player $1$ moves at $v_{0,1}$. Then Player $2$ counters by making a move at $v_{0,2+k}$. This causes vertex $v_{1,2}$ to become blocked. As a result $\omega^{(2)}(v_{1,1})=\omega^{(2)}(v_{1,2+k})=1$, rendering four blocked vertices, viz.\ $v_{1,1+k}$, $v_{1,1+2k}$, $v_{1,2}$, $v_{1,2+2k}$. Repeating this strategy, Player $2$ can always guarantee that no three consecutive outer vertices have weight $1$ irrespective of Player $1$'s moves. This proves Player $2$ has a winning strategy, in which case $\mathcal{G}(\mathcal{P}_{1,0}(3k,k))=0$.
\end{proof}
    
\begin{theorem}\label{thm:gp_nim_bounds}
    For $0 < k \le \lfloor\frac{m-1}{2}\rfloor$
    \begin{enumerate}
        \item $\mathcal{G}(\mathcal{P}_{1,1}(m,k)) \in \{0,1,2\}$,
        \item $\mathcal{G}(\mathcal{P}_{0,1}(m,k)),\; \mathcal{G}(\mathcal{P}_{1,0} (m,k)) \in \{0,1\}$.
    \end{enumerate}
\end{theorem}

\begin{proof}

    In the case of $\mathcal{P}_{1,1}(m,k)$, there are only two possible initial moves up to symmetry. By the manner in which Nimbers are defined, the Nimber cannot be greater than $2$.

    In the cases of $\mathcal{P}_{0,1}(m,k)$ and $\mathcal{P}_{1,0}(m,k)$, there is only one possible initial move up to symmetry, so the Nimber cannot be greater than $1$.
\end{proof}

To more efficiently determine Nimbers $\mathcal{G}(\mathcal{P}_{0,1}(m,k))$ and $\mathcal{G}(\mathcal{P}_{1,0}(m,k))$, $1 \leq k \leq 2$, we introduce \emph{Jacob's Ladder} $JL_m$, a two-player impartial game played on an $m$-cycle $C_m$. A move here consists of choosing a vertex $v$ on $C_m$ and removing all vertices in $N_2[v]$. The player with no available legal moves loses the game.

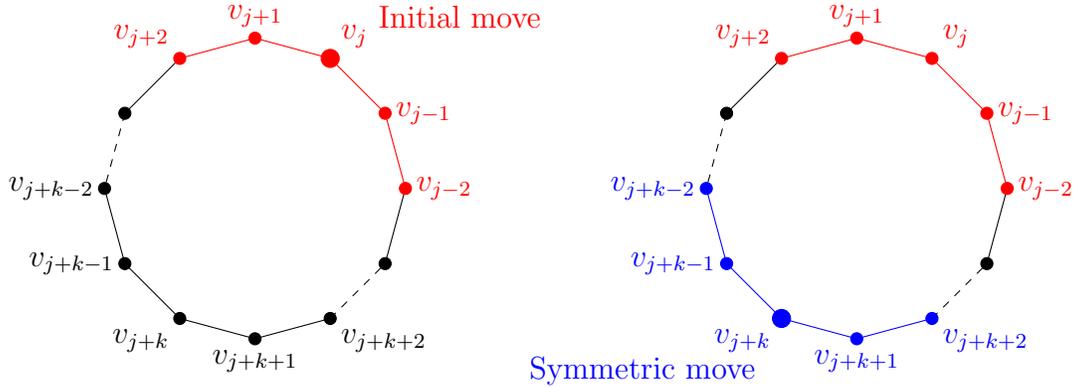
\begin{figure}[H]
    \centering
    \begin{tikzpicture}
    \coordinate(a1) at (-2,0);
    \coordinate(a2) at (-2.27,1);
    \coordinate(a3) at (-3,1.73);
    \coordinate(a4) at (-4,2);
    \coordinate(a5) at (-5,1.73);
    \coordinate(a6) at (-5.73,1);
    \coordinate(a7) at (-6,0);
    \coordinate(a8) at (-5.73,-1);
    \coordinate(a9) at (-5,-1.73);
    \coordinate(a10) at (-4,-2);
    \coordinate(a11) at (-3,-1.73);
    \coordinate(a12) at (-2.27,-1);

    \draw[red](a1)--(a2)--(a3)--(a4)--(a5);
    \draw(a5)--(a6);
    \draw[dashed](a6)--(a7);
    \draw(a7)--(a8)--(a9)--(a10)--(a11);
    \draw[dashed](a11)--(a12);
    \draw(a12)--(a1);

    \foreach\i in{a3}{\filldraw[red](\i)circle(0.12);}
    \foreach\i in{a1,a2,a4,a5}{\filldraw[red](\i)circle(0.08);}
    \foreach\i in{a6,a7,a8,a9,a10,a11,a12}{\filldraw(\i)circle(0.08);}

    \coordinate(b1) at (6,0);
    \coordinate(b2) at (5.73,1);
    \coordinate(b3) at (5,1.73);
    \coordinate(b4) at (4,2);
    \coordinate(b5) at (3,1.73);
    \coordinate(b6) at (2.27,1);
    \coordinate(b7) at (2,0);
    \coordinate(b8) at (2.27,-1);
    \coordinate(b9) at (3,-1.73);
    \coordinate(b10) at (4,-2);
    \coordinate(b11) at (5,-1.73);
    \coordinate(b12) at (5.73,-1);

    \draw[red](b1)--(b2)--(b3)--(b4)--(b5);
    \draw(b5)--(b6);
    \draw[dashed](b6)--(b7);
    \draw[blue](b7)--(b8)--(b9)--(b10)--(b11);
    \draw[dashed](b11)--(b12);
    \draw(b12)--(b1);

    \foreach\i in{b1,b2,b3,b4,b5}{\filldraw[red](\i)circle(0.08);}
    \foreach\i in{b9}{\filldraw[blue](\i)circle(0.12);}
    \foreach\i in{b7,b8,b10,b11}{\filldraw[blue](\i)circle(0.08);}
    \foreach\i in{b6,b12}{\filldraw(\i)circle(0.08);}
    
\node[right,red]at(a1){$v_{j-2}$};
\node[right,red]at(a2){$v_{j-1}$};
\node[above right,red]at(a3){$v_{j}$};
\node[above,red]at(a4){$v_{j+1}$};
\node[above left,red]at(a5){$v_{j+2}$};
\node[left]at(a7){$v_{j+k-2}$};
\node[left]at(a8){$v_{j+k-1}$};
\node[below left]at(a9){$v_{j+k}$};
\node[below]at(a10){$v_{j+k+1}$};
\node[below right]at(a11){$v_{j+k+2}$};
\node[above right,red] at(-2.5,2){Initial move};

\node[right,red]at(b1){$v_{j-2}$};
\node[right,red]at(b2){$v_{j-1}$};
\node[above right,red]at(b3){$v_{j}$};
\node[above,red]at(b4){$v_{j+1}$};
\node[above left,red]at(b5){$v_{j+2}$};
\node[left, blue]at(b7){$v_{j+k-2}$};
\node[left, blue]at(b8){$v_{j+k-1}$};
\node[below left, blue]at(b9){$v_{j+k}$};
\node[below, blue]at(b10){$v_{j+k+1}$};
\node[below right, blue]at(b11){$v_{j+k+2}$};
\node[below left,blue] at(2.8,-2.1){Symmetric move};

\end{tikzpicture}
\caption{A symmetric move is played at $v_{j+k}$ in response to an initial move at $v_j$.}
\label{fig:symmetric}
\end{figure}


Denote by $JL_k + JL_k$ the game of Jacob's Ladder played on $C_k + C_k$.

\begin{lemma}\label{lem:JL_2k}
    For $k \ge 3$, $\mathcal{G}(JL_{2k}) = \mathcal{G}(JL_k + JL_k) = 0$.
\end{lemma}

\begin{proof}
    Since $k \geq 3$, the game $JL_{2k}$ does not end after a single move. Because there are $2k$ vertices, for every move made by the first player there is a \emph{symmetric move} that can be made by the second player (See Figure \ref{fig:symmetric}). This results in a winning strategy for the second player, i.e.\ $\mathcal{G}(JL_{2k}) = 0$.

    In the case of $JL_k + JL_k$, Player $2$ has a winning strategy by simply mirroring each of Player $1$'s moves, that is, playing in the opposite cycle at the vertex that corresponds to Player $1$'s move.
\end{proof}

\begin{lemma}\label{lemma:unplayable_GPn1}
In a game of Toggle on $\mathcal{P}_{0,1}(m,1)$ or $\mathcal{P}_{1,0}(m,1)$, if a vertex $v$ becomes unplayable at stage $j$ then $v$ is terminally unplayable at stage $j$, i.e.\ vertex $v$ is terminally unplayable once it becomes unplayable.
\end{lemma}
\begin{proof}
We prove this for the case $\mathcal{P}_{0,1}(m,1)$, since the case for $\mathcal{P}_{1,0}(m,1)$ is entirely symmetric. Proposition \ref{prop:degree3penultimately} implies that in a game of Toggle on $C_m$ with initial assignment $V_0^{(0)}=\emptyset$, any vertex $v_i$ is terminally unplayable once it becomes unplayable. 

    The outer cycle of $\mathcal{P}_{0,1}(m,1)$ may be thought of as a Toggle game on $C_m$ with $V_0^{(0)}=\emptyset$. So for each outer vertex $v_{1,i}$, being unplayable is equivalent to being terminally unplayable unless some inner vertex $v_{0,j}$ affects the playability of $v_{1,i}$.

    Suppose the initial move is made at $v_{1,i}$. Then $\omega^{(1)}(v_{1,i-1})=\omega^{(1)}(v_{1,i})=\omega^{(1)}(v_{1,i+1})=0$ and $\omega^{(1)}(v_{0,i})=1$. Since $v_{1,i-1}$ is terminally unplayable in $C_m$, it can only become playable if $v_{0,i-1}$ has weight $1$ at some stage. But this can occur only if a Toggle move is made on $v_{1,i-1}$ or on some inner vertex. Since Toggle moves must be made on at least two adjacent outer vertices for an inner vertex to become playable, and since $v_{1,i-1}$ is unplayable, it follows that $v_{1,i-1}$ is terminally unplayable. (Observe that the argument for $v_{1,i+1}$ is symmetric.) This further implies that all inner vertices are terminally unplayable at the initial stage, since no two adjacent outer vertices can both be played within the same game. Because all of the neighbors of $v_{1,i}$ are terminally unplayable and $\omega^{(1)}(v_{1,i})=0$, $v_{1,i}$ is also terminally unplayable.

    Finally, observe that the vertex $v_{1,i-2}$ is unplayable after the initial move, and $v_{1,i-2}$ can only become playable if either $\omega^{(\ell)}(v_{0,i-2})=1$ or $\omega^{(\ell)}(v_{1,i-1})=1$ at some stage $\ell$. But since all inner vertices are terminally unplayable and $v_{1,i}$ is terminally unplayable, these can only occur if a Toggle move is played at $v_{1,i-2}$. Therefore $v_{1,i-2}$ is terminally unplayable. (The argument for $v_{1,i+2}$ is symmetric.) Thus any vertex $v$ which becomes unplayable at the initial move also becomes terminally unplayable. Because of this and the structure of $\mathcal{P}_{0,1}(m,1)$, all future moves result in the same pattern as that obtained by the initial move. Therefore any vertex $v$ in $\mathcal{P}_{0,1}(m,1)$ is terminally unplayable the moment it becomes unplayable.
\end{proof}

We next show the relationship between the Nimbers of Jacob's Ladder and those of Toggle.

\begin{lemma}\label{lem:JL_GPm1}
    $\mathcal{G}(JL_m) = \mathcal{G}(\mathcal{P}_{0,1}(m,1)) = \mathcal{G}(\mathcal{P}_{1,0}(m,1))$.
\end{lemma}
\begin{proof}
By Corollary \ref{cor:isomorphism}(2), $\mathcal{G}(\mathcal{P}_{0,1}(m,1)) = \mathcal{G}(\mathcal{P}_{1,0}(m,1))$, so it suffices to show $\mathcal{G}(JL_m) = \mathcal{G}(\mathcal{P}_{0,1}(m,1))$. Thus we consider a game of Toggle with initial position $\mathcal{P}_{0,1}(m,1)$. Let $G$ be the subgraph whose vertex set consists of all playable vertices of $P(m,1)$. By Lemma \ref{lemma:unplayable_GPn1}, all inner vertices are terminally unplayable, hence $G$ is initially the cycle $C_m$ comprised of the outer vertices of $P(m,1)$. Moreover, any vertex $v$ is terminally unplayable the moment it becomes unplayable. It follows that any vertex $v$ removed from $G$ remains removed for the balance of the game. Since a move at $v$ renders all vertices within distance $2$ of $v$ unplayable, we conclude that the game is equivalent to $JL_m$.
\end{proof}

\begin{theorem}
    For $k \geq 3$, $\mathcal{G}(\mathcal{P}_{0,1}(2k,1))=0$.
\end{theorem}
\begin{proof}
    This follows immediately from Lemmas \ref{lem:JL_2k} and \ref{lem:JL_GPm1}.
\end{proof}

An \emph{octal game} is an impartial ``take and break'' game that involves removing beans from heaps of beans \cite{BCG}. Each octal game has a specific octal code $$\cdot {\bf d}_1 {\bf d}_2 {\bf d}_3 \dotsc$$ which specifies the set of permissible moves in the take and break game. In this code, the $k$-th digit ${\bf d}_k$ is the sum of a (possibly empty) subset of $\{{\bf 1},{\bf 2},{\bf 4}\}$, where 
\begin{enumerate}
    \item[\{{\bf 1}\}] indicates that a heap can be completely removed by removing $k$ beans;
    \item[\{{\bf 2}\}] indicates that a heap can be reduced in size by removing $k$ beans; and
    \item[\{{\bf 4}\}] indicates that a heap can be split into two heaps of smaller respective sizes by removing $k$ beans.
\end{enumerate}
As an example, the octal game $\cdot {\bf 356}$ has ${\bf d}_1 = 3$, ${\bf d}_2 = 5$, and ${\bf d}_3 = 6$. This code stipulates that there are three options:
\begin{enumerate}
    \item ${\bf d}_1=3=1+2$ indicates that a player can remove one bean from a heap to either \{{\bf 1}\} completely remove that heap or \{{\bf 2}\} reduce the size of that heap.
    \item  ${\bf d}_2 = 5 = 1 + 4$ indicates that a player can remove two beans from a heap to either \{{\bf 1}\} completely remove that heap or \{{\bf 4}\} split that heap into two smaller heaps.
    \item ${\bf d}_3 = 6 = 2 + 4$ indicates that a player can remove three beans from a heap to either \{{\bf 2}\} reduce the size of that heap or \{{\bf 4}\} split that heap into two smaller heaps.
\end{enumerate}

The following result links Jacob's Ladder to an octal game. The proof is straightforward, so is left to the reader.

\begin{proposition}\label{prop:JL_octal}
    After any single move, Jacob's Ladder becomes the octal game $\cdot 11337$.
\end{proposition}

    \begin{remark}
    The Nimbers of octal game $\cdot 11337$ can be found in the On-Line Encyclopedia of Integer Sequences \cite{oeis}, see entry A071426. By Lemma \ref{lem:JL_GPm1} and Proposition \ref{prop:JL_octal}, the Toggle game $\mathcal{P}_{0,1}(m,1)$ becomes equivalent to the octal game after any initial Toggle move. This explains why the Nimbers of $\mathcal{P}_{0,1}(m,1)$ can be obtained from those of the octal game $\cdot 11337$ by removing the first three entries of sequence A071426 and changing each positive entry to $0$ and each $0$ to $1$. (The fact that each $0$ is changed to $1$ rather than to another positive integer follows from part 2 of Theorem \ref{thm:gp_nim_bounds}.)
\end{remark}

\begin{lemma}\label{P01(m,2)}
    In a game of Toggle on $\mathcal{P}_{0,1}(m,2)$, vertex $v$ is terminally unplayable once it becomes unplayable.
\end{lemma}

\begin{proof}
    Proposition \ref{prop:degree3penultimately} implies that in a game of Toggle on $C_m$ with initial assignment $V_0^{(0)}=\emptyset$, any vertex $v_i$ is terminally unplayable once it becomes unplayable. 

    The outer cycle of $\mathcal{P}_{0,1}(m,2)$ may be thought of as a Toggle game on $C_m$ with $V_0^{(0)}=\emptyset$. So for each outer vertex $v_{1,i}$, being unplayable is equivalent to being terminally unplayable unless some inner vertex $v_{0,j}$ affects the playability of $v_{1,i}$.

    Suppose the initial move is made at $v_{1,i}$. Then $\omega^{(1)}(v_{1,i-1})=\omega^{(1)}(v_{1,i})=\omega^{(1)}(v_{1,i+1})=0$ and $\omega^{(1)}(v_{0,i})=1$. The vertex $v_{1,i-2}$ can only become playable if either $\omega^{(\ell)}(v_{0,i-2})=1$ or $\omega^{(\ell)}(v_{1,i-1})=1$ at some stage $\ell$. Note that for an inner vertex to become playable, Toggle moves must be made on at least two outer vertices that are distance $2$ apart. Thus since $v_{1,i-2}$ is unplayable, it follows that $v_{1,i-2}$ is terminally unplayable. (Note that the argument for $v_{1,i+2}$ is symmetric.) This implies that all inner vertices are terminally unplayable at the initial stage since no two outer vertices that are distance two apart can both be played within the same game.

    Finally, for the vertex $v_{1,i-1}$ to become playable, we must have $\omega^{(\ell)}(v_{1,i-1})=1$ at some stage $\ell$. This can only occur if a Toggle move is made on $v_{1,i-2}$, $v_{1,i}$, or $v_{0,i-1}$. Since $v_{1,i-2}$ and all inner vertices are terminally unplayable, a move must be made on $v_{1,i}$. However, $v_{1,i}$ can only become playable if $\omega^{(\ell)}(v_{1,i-1})=1$ or $\omega^{(\ell)}(v_{1,i+1})=1$ at some stage $\ell$. Therefore $v_{1,i}$ and $v_{1,i-1}$ (and by symmetry, $v_{1,i+1}$) are terminally unplayable. Thus any vertex $v$ which becomes unplayable at the initial move also becomes terminally unplayable. Because of this and the structure of $\mathcal{P}_{0,1}(m,2)$, all future moves result in the same pattern as that obtained by the initial move. Therefore any vertex $v$ in $\mathcal{P}_{0,1}(m,2)$ is terminally unplayable the moment it becomes unplayable.
\end{proof}

\begin{lemma}\label{P10(m,2)}
    In a game of Toggle on $\mathcal{P}_{1,0}(m,2)$, vertex $v$ is terminally unplayable once it becomes unplayable.
\end{lemma}

\begin{proof}
    Proposition \ref{prop:degree3penultimately} implies that in a game of Toggle on $C_m$ with initial assignment $V_0^{(0)}=\emptyset$, any vertex $v_i$ is terminally unplayable once it becomes unplayable. The inner star polygon $\{m,2\}$ of $\mathcal{P}_{1,0}(m,2)$ may be thought of as a Toggle game on $C_m$ (or two copies of $C_{m/2}$ if $m$ is even) with $V_0^{(0)}=\emptyset$. So for each inner vertex $v_{0,i}$, being unplayable is equivalent to being terminally unplayable unless some outer vertex affects the playability of $v_{0,i}$.
    
    Suppose the initial move is made at $v_{0,i}$. Then $\omega^{(1)}(v_{0,i-2})=\omega^{(1)}(v_{0,i})=\omega^{(1)}(v_{1,i+2})=0$ and $\omega^{(1)}(v_{1,i})=1$. Since $v_{0,i-2}$ is terminally unplayable in its inner cycle within the inner star polygon, it can become playable only if $\omega^{(\ell)}(v_{1,i-2})=1$ at some stage $\ell$. Since $v_{0,i-2}$ is unplayable, this can only occur if a move is made on an outer vertex. One of the following two cases must occur for the first outer vertex to become playable.
    
    \noindent \textbf{Case 1:} $\omega^{(\ell)}(v_{1,j-2})=\omega^{(\ell)}(v_{1,j-1})=\omega^{(\ell)}(v_{1,j})=1$ at some stage $\ell$. For this to occur, moves must be made at inner vertices $v_{0,j}$, $v_{0,j-1}$, and $v_{0,j-2}$. But if a move is made on $v_{0,j}$, then a move cannot be made at $v_{0,j-2}$ prior to a move at another outer vertex. Thus, this case cannot arise. 

    \noindent \textbf{Case 2:} $\omega^{(\ell)}(v_{1,j-1})=\omega^{(\ell)}(v_{1,j})=\omega^{(\ell)}(v_{0,j})=1$ at some stage $\ell$. For this to occur, moves must be made at $v_{0,j}$ and $v_{0,j-1}$, followed by a move at an inner neighbor of $v_{0,j}$. But a move cannot be made at $v_{0,j-2}$ or $v_{0,j+2}$ prior to a move at another outer vertex, so this case is impossible as well.

    Since neither of these cases is possible, it follows that all outer vertices are terminally unplayable at the initial stage. This further implies that $v_{0,i-2}$ is terminally unplayable. (The argument for $v_{0,i+2}$ is symmetric.) Since $\omega^{(1)}(v_{0,i})=0$ and all neighbors of $v_{0,i}$ are terminally unplayable, $v_{0,i}$ is also terminally unplayable.

    Finally, the vertex $v_{0,i-4}$ can only become playable if either $\omega^{(\ell)}(v_{0,i-2})=1$ or $\omega^{(\ell)}(v_{1,i-4})=1$ at some stage $\ell$. Since all outer vertices are terminally unplayable and $v_{0,1}$ is terminally unplayable, this is only possible if a move is made at $v_{0,i-4}$. We conclude that $v_{0,i-4}$ is terminally unplayable. (By symmetry, $v_{0,i+4}$ is also terminally unplayable.) Thus any vertex $v$ which becomes unplayable at the initial move also becomes terminally unplayable. Because of this and the structure of $\mathcal{P}_{1,0}(m,2)$, all future moves result in the same pattern as that obtained by the initial move. Therefore, any vertex $v$ in $\mathcal{P}_{1,0}(m,2)$ is terminally unplayable the moment it becomes unplayable.
\end{proof}

\begin{lemma}\label{lem:JL_GPm2}
    $\mathcal{G}(JL_m)=\mathcal{G}(\mathcal{P}_{0,1}(m,2))=\mathcal{G}(\mathcal{P}_{1,0}(m,2))$.
\end{lemma}

\begin{proof}
    Consider a game of Toggle with initial position $\mathcal{P}_{0,1}(m,2)$ and let $G$ be the subgraph whose vertex set consists of all playable vertices of $P(m,2)$. By Lemma \ref{P01(m,2)}, all inner vertices are terminally unplayable, so $G$ is initially the cycle $C_m$ comprised of the outer vertices of $P(m,2)$. Moreover, any vertex $v$ is terminally unplayable the moment it becomes unplayable, so whenever a vertex $v$ is removed from $G$ it remains removed for the entire game. Since a Toggle move at vertex $v$ renders all vertices within distance $2$ of $v$ unplayable, we conclude that the game is equivalent to $JL_m$.

    Now consider a game of Toggle with initial position $\mathcal{P}_{1,0}(m,2)$ and again let $G$ be the subgraph whose vertex set consists of all playable vertices of $P(m,2)$. By Lemma \ref{P10(m,2)}, all outer vertices are terminally unplayable, so $G$ is initially the inner star polygon $\{m,2\}$. If $m$ is odd, $G$ is simply the cycle $C_m$, whereas if $m$ is even $G$ is two disjoint copies of $C_{m/2}$. Again, any vertex $v$ is terminally unplayable the moment it becomes unplayable, so a vertex removed from $G$ remains removed for the entire game. Thus the game is equivalent to $JL_m$ if $m$ is odd or to two disjoint copies of $JL_{m/2}$ if $m$ is even. By Lemma \ref{lem:JL_2k}, $\mathcal{G}(JL_m)=\mathcal{G}(JL_{m/2} + JL_{m/2})$ for $m$ even, so in either case $\mathcal{G}(\mathcal{P}_{1,0}(m,2)) = \mathcal{G}(JL_m)$. We conclude that $\mathcal{G}(JL_m)=\mathcal{G}(\mathcal{P}_{0,1}(m,2))=\mathcal{G}(\mathcal{P}_{1,0}(m,2))$.
\end{proof}

\begin{theorem}\label{thm:gpm1_m2}
    $\mathcal{G}(\mathcal{P}_{0,1}(m,1))=\mathcal{G}(\mathcal{P}_{1,0}(m,1)) = \mathcal{G}(\mathcal{P}_{0,1}(m,2))=\mathcal{G}(\mathcal{P}_{1,0}(m,2))$.
\end{theorem}
\begin{proof}
    This follows immediately from Lemmas \ref{lem:JL_GPm1} and \ref{lem:JL_GPm2}.
\end{proof}

Tables \ref{tab:NimbersP01} and \ref{tab:NimbersP10} in Appendix B reflect Theorem \ref{thm:gpm1_m2}. See also entry A361517 in the On-line Encyclopedia of Integer Sequences \cite{oeis}.



\section{Quantified Constraint Logic}

We take the following standard definitions from \cite{LS77}:

\begin{definition}
    A propositional formula is in \emph{conjunctive normal form (CNF)} provided it consists of a conjunction of disjunctions of literals. This is often restated as being intersections of unions.
\end{definition}

\begin{definition}
     Let ${\rm{SPACE}}(n^k)$ be the class of languages accepted by deterministic Turing machines within space $n^k$. The class of languages $\rm{PSPACE}$ is defined as
     $${\rm{PSPACE}} = \bigcup\limits_{k=1}^{\infty}{\rm{SPACE}}(n^k)$$
\end{definition}

\begin{definition}
    By \emph{logspace} we refer to the class of functions computable by deterministic Turing machines within space $log(n)$.
\end{definition}

\begin{definition}
    Let $\Theta$ and $\Delta$ be finite alphabets. We define $\Theta^+$ and $\Delta^+$ to be finite strings in the alphabets $\Theta$ and $\Delta$, respectively. For $A\subseteq \Theta^{+}$ and $B\subseteq \Delta^{+}$, we let $f: \Theta^+ \rightarrow \Delta^+$ be a transformation with $f\in \text{logspace}$ such that $x\in A$ if and only if $f(x) \in B$ for all $x\in \Theta^{+}$. In such case we say \emph{$A\subseteq \Theta^{+}$ transforms (i.e.\ reduces to) $B\subseteq \Delta^{+}$ within logspace via $f$}, and denote this by $A\leq_{log} B$ via $f$.
\end{definition}

We describe a \emph{Quantified Boolean Formula (QBF) decision problem} as follows: Given a set $\widehat \beta = \{\beta_1, ..., \beta_n\}$ of Boolean variables and a quantified Boolean $CNF$ formula $\varphi$, decide whether or not $\varphi$ evaluates to True.

We refer to a $QBF$ with exactly $3$ variables in each clause of the $CNF$ formula as $3$-$QBF$. The following result appears in \cite{DD00}.

\begin{lemma}\label{lem:3QBF}
$3$-$QBF$ is PSPACE-complete.
\end{lemma}


Schaefer \cite{TS78} generalizes the $QBF$ problem to an impartial, two-player game in the following manner: A game instance of $QBF$, designated as $G_\omega (QBF) $, takes as input a set of indexed Boolean variables and a formula, as above. The game consists of two players who take turns assigning values to the variables, e.g.\ Player $1$ assigns a value to variable $\beta_1$, Player $2$ assigns a value to $\beta_2$, and so on until a value is assigned to $\beta_n$ ending the game. (Note that the player who assigns the variable $\beta_n$ depends on the parity of $n$.) Player $1$ wins if and only if the formula $\varphi$ evaluates to True. In logical terms, this is expressed as $$(\exists \beta_1)\;( \forall\beta_2) \; (\exists \beta_3) \; \dots \; (\exists\beta_n \text{ or }\forall\beta_n) : \varphi.$$ Schaefer also extends Lemma \ref{lem:3QBF} to $G_\omega(QBF)$ and $G_\omega(3$-$QBF)$.

Stockmeyer goes on to prove the following result in \cite{LS77}.

\begin{proposition}\label{prop:log-complete}
    Let $\vartheta$ be a set and $\varepsilon$ be a class of sets. Then $\vartheta$ is log-complete in $\varepsilon$ if and only if there exists a function $f$ such that $\varepsilon \leq_{log} \vartheta$ via $f$ where $\vartheta\in\varepsilon$.    
\end{proposition}

Intuitively this says, if  every problem that is PSPACE-complete can be logspace reduced from some problem $\vartheta\in \rm{PSPACE}$, then $\vartheta$ is log-complete in PSPACE.

We now provide a context for applying the above results to the spatial complexity of Toggle. Framework for the following proposition is taken from Schaefer \cite{TS78}.

\begin{proposition}\label{prop:PSPACE}
    Given any position $\mathcal{S}$ in a Toggle game there exists a polynomial space algorithm for determining the winner, i.e.\ the game of Toggle is in PSPACE.
\end{proposition}
\begin{proof}
    Let $\mathcal{S}$ be an arbitrary position in a game of Toggle played on the graph $G$. Because the maximum weight of each vertex is one, we know that at each stage $j$ the total weight $\sigma^{(j)}(G)$ is at most $|V(G)|-j$. It follows that each game lasts for at most $|V(G)|$ moves.

    Starting from position $\mathcal{S}$, let $\mathcal{S}_\alpha$ be the position reached by playing any sequence $\alpha$ of legal moves. (Here we allow $\alpha$ to be the empty sequence.)
    
    Let $Succ(\alpha)$ denote the set of legal moves playable at position $\mathcal{S}_\alpha$, and let $\alpha m$ be the sequence of moves $\alpha$ followed by the move $m \in Succ(\alpha)$.
    Consider the following recursive algorithm: if $S_\alpha$ is a completed game, and thus $Succ(\alpha) = \emptyset$, then whichever player played last is the winner, denoted $Winner(\alpha)$. Otherwise, there are remaining legal moves (i.e.\ $Succ(\alpha)$ is not empty) and for all $m \in Succ(\alpha)$ there exists a $Winner(\alpha m)$ found by recursive computation. If $Winner(\alpha m) = \rho$ for every $m \in Succ(\alpha)$, then $Winner(\alpha) = \rho$. On the contrary, if there exists a strategy for move $m \in Succ(\alpha)$ such that $\rho \neq Winner(\alpha m)$, then $\rho \neq Winner(\alpha)$.

    Using the above algorithm, the space needed to store any legal sequence of moves is bounded above by the number of vertices in the graph. In addition, the space required to decide if a sequence represents a finished game is also bounded above by the number of vertices. The total computational space needed to determine $Winner(\emptyset)$ is the sum of these two values. Because this algorithm can be performed using an amount of space that is at most polynomial with respect to the length of input, we have that $\rm{Toggle} \in PSPACE$.
\end{proof}

An immediate consequence of Proposition \ref{prop:PSPACE} is the following.
\begin{corollary}
    The Nimber $\mathcal{G}(\mathcal{S})$ corresponding to position $\mathcal{S}$ of a Toggle game can be computed in polynomial space with respect to input.
\end{corollary}

Our aim is to show that a known PSPACE-complete problem, specifically $3$-$QBF$, is polynomially equivalent to Toggle with respect to computational space complexity. Following Schaeffer \cite{TS78}, our method is to construct a general function that equates the satisfiability of a $3$-$QBF$ decision problem with the outcome of a corresponding Toggle game. We have already shown that $\rm{Toggle} \in \rm{PSPACE}$ in Proposition \ref{prop:PSPACE}. Thus by Proposition \ref{prop:log-complete}, it suffices to show that a given input of $3$-$QBF$ reduces to Toggle within logspace, i.e.\ $G_\omega(3\text{-}QBF)\leq_{log} G_\omega(\rm{Toggle})$. We show this by first proving that for every $3\text{-}QBF$ game there exists an instance of a Toggle game such that a winning strategy in one game is equivalent to a winning strategy in the other. This allows us to formally define the Toggle space complexity class in terms of the marginal size of a Toggle graph with respect to input.


Starting with a generalized instance of the $3$-$QBF$ decision problem with $m$ clauses and $n$ variables, we construct a logically equivalent Toggle game. Much of the remainder of this section is dedicated to developing such a Toggle instance and proving its logical equivalence to the underlying $3$-$QBF$ input.

We adopt the notation $\gamma^{\delta}_i$ for the labeling of vertices, where $\delta$ denotes subtype, $i$ denotes an individual identifier, and $\gamma \in \{d, c, \sigma, v, \chi, \lambda\}$ denotes the vertex type as defined below. See Figure \ref{fig:log-complete ref} as a helpful visual reference.

\begin{itemize}
    \item $d$ stands for dummy vertices. A dummy vertex is never playable and only serves to ensure whether or not a neighbor of that vertex becomes playable. Specifically, the dummy vertex of subtype $\delta$ indicates the neighborhood of that vertex.
    \item $c$ stands for controller vertices. Each controller vertex $c^1_i$ is never playable and serves to ensure that variable vertices are toggled in the correct order. Each controller vertex $c^2_i$ becomes playable only after all variables in clause $i$ have been assigned. Note that there is a unique controller vertex $c^2_i$ for each clause $i$, $1 \leq i \leq m$.

    \item $\sigma$ stands for signal vertices. Each signal vertex $\sigma^1_i$ is connected to a specific variable truth assignment. It is never playable and serves to reflect variable assignments. Each signal vertex $\sigma^2_i$ is connected to a clause vertex and signals to the next clause when the previous clause vertex has been toggled.
    \item $v$ stands for variable vertices. These are the only vertices at which there is an option to play. Toggling $v^0_i$ sets $\beta_i$ to False while toggling $v^1_i$ sets $\beta_i$ to True. Note that exactly one of $v_i^0$ and $v_i^1$ must be played at stage $i$. 
    \item $\chi$ stands for clause vertices. The playability of clause vertex $\chi_i$ is determined by whether the variable assignments imply that the clause $\chi_i$ returns a value of True. A clause vertex $\chi_i$ will be toggled if and only if at least one variable in the clause has been assigned the value True and $\chi_j$ has been toggled for all $j<i$. The $QBF$ formula $\varphi$ is satisfied (i.e.\ Player $1$ wins) if and only if all clause vertices have been toggled.

    \item $\lambda$ stands for link vertices. The link vertices separate the clause and variable vertices and are played after all variable values have been assigned (True or False) and before any clause vertices have been toggled. Note that there will be two link vertices if $n$ is even and three link vertices if $n$ is odd (cf.\ proof of Theorem \ref{thm:Toggle_Pspace}).
\end{itemize}

We are nearly prepared to provide a proof of our main result on complexity. Prior to this, it is convenient to lay the following framework.


Let $A=\big(\widehat \beta, \varphi\big)$ be a given input for a $3\text{-}QBF$ game. Without loss of generality, we may assume that
$$A=(\exists \beta_1)\;( \forall\beta_2) \; (\exists \beta_3) \; \dots \; (\exists\beta_n \text{ or }\forall\beta_n)(\chi_1 \wedge \chi_2 \wedge ... \wedge \chi_m)$$
where $\chi_i$ is a $CNF$ clause with three variables.



Let $\{G,V_0^{(0)}\}$ be the Toggle position associated with $A$. Let $Tree_V(n)=\emptyset$ if $n$ is even and $Tree_V(n)=\{\lambda_3,d_1^9,d_2^9,d_1^{10},d_2^{10},d_3^{10},d_4^{10}\}$ if $n$ is odd. We define the vertex set of $G$ as follows:
\begin{align*}
    V(G) = &\left\{v^0_j, v^1_j, c^1_j, d^5_j \mid 1\leq j\leq n\right\} \cup \left\{ \chi_i, c^2_i, d^6_i, d^7_i, d^{13}_i \mid 1 \leq i \leq m \right\} \\&\cup\left\{d^4_{i},d^{14}_i \mid 1\leq i \leq 2m\right\}  \cup \left\{ d^3_{i}, \sigma^1_{i}, \sigma^2_{i}  \mid 1 \leq i \leq 3m \right\} \cup \left\{d^2_{i} \mid 1 \leq i \leq 6m\right\} \\ &\cup \left\{d^8_{i}\mid 1\leq i \leq 4m \right\} \cup \left\{d^{11}_1\right\} \cup \left\{d^{12}_1,d^{12}_2\right\} \cup \left\{d^1_1,d^1_2,d^1_3\right\} \\&\cup \{\lambda_1,\lambda_2\} \cup Tree_V(n) \cup \{Endgame\}.
\end{align*} For $n$ even, the assignment on $G$ is given by \vspace{-2mm} $$V_{1}^{(0)}(G) = \{v_1^i,v_2^i,d_i^3 \mid 1 \leq i \leq n\} \cup \{\chi_j \mid 1 \leq j \leq m\} \cup \{\lambda_2, d_1^{11}, c_1^1\} \cup \{d_j^8 \mid 1 \leq j \leq 4m\},$$
\vspace{-7mm} 

\noindent whereas for $n$ odd, $V_{1}^{(0)}(G)$ additionally includes $\{d_1^9,d_2^9\}$.

Before defining the edge set $E(G)$, some additional definitions are in order.

Let $\c{C}^0$ denote the set of ordered pairs $(\beta_i, \chi_j)$ such that $\beta_i$ appears in a negated form, $\neg\beta_i$, in $\chi_j$ (i.e.\ $\neg \beta_i \implies \chi_j$). Similarly, let $\c{C}^1$ denote the set of ordered pairs $(\beta_i, \chi_j)$ such that the non-negated form of $\beta_i$ appears in $\chi_j$ (i.e.\ $ \beta_i \implies \chi_j$). Finally, let $\c{C}^* = \c{C}^0 \cup \c{C}^1$. We also define
$$\c{C}^0(z) = \left\{ (\beta_i, \chi_j) \mid \neg \beta_i \implies \chi_j \textbf{ and } i<z\right\} $$
$$\c{C}^1(z) = \left\{ (\beta_i, \chi_j) \mid \beta_i \implies \chi_j \textbf{ and } i<z\right\} $$ 
$$\c{C}^{*}(z) = \c{C}^0(z) \cup \c{C}^1(z)$$

Note that multiplicities are preserved, i.e.\ if some $\beta_i$ appears in a clause $\chi_j$ three times then $(\beta_i, \chi_j)$ is in $\c{C}^{*}$ three times. Thus, as an immediate consequence, $|\c{C}^*|=3m$. 

 For each $j$, we order $i_1 \leq i_2 \leq i_3$ such that $(\beta_{i_k},\chi_j) \in \c{C}^*$ for $k=1,2,3$. Then

$$y_1: \c{C}^* \to \left\{1, 2, \dots, 3m \right\}$$ 
$$y_1(\beta_{i_1}, \chi_j) = 3j-2, \;y_1(\beta_{i_2}, \chi_j) = 3j-1, \;y_1(\beta_{i_3}, \chi_j) = 3j$$
$$y_2: \widehat \beta \to A\subseteq \c{C}^*$$ 
$$y_2(\beta_i) = \left\{ (\beta_k, \chi_{j}) \in \c{C}^* \mid k=i \right\}.$$
Observe that $y_1$ is a bijection. Further, $y_2$ partitions $\c{C}^*$, as indicated below:
$$\bigcup_{\ell = 1}^n y_2(\beta_\ell) = \c{C}^* ~\text{\;  and  \;}~ y_2(\beta_\ell) \cap y_2(\beta_k) = \emptyset ~~~ \forall \ell \neq k.$$

We are now prepared to define the edge set

$$E(G) = \textcolor{red}{R(G)} \cup \textcolor{blue}{B(G)} \cup \textcolor{violet}{P(G)}.$$
The use of colors in the above notation is a device to help the reader to better interpret Figures \ref{fig:log-complete ref}-\ref{fig:PSpaceMargCla}. Here $\textcolor{red}{R(G)}, \textcolor{blue}{B(G)}, \textcolor{violet}{P(G)}$ are defined as follows:

\begin{align*}
    \textcolor{red}{R(G)} = &\left\{(v_n^0,\lambda_1), (v_n^1, \lambda_1), (\lambda_1, \lambda_2), (\lambda_2, d_1^{11}), (d_1^{11}, d_1^{12}), (d_1^{11}, d_2^{12})\right\}  \\ & \cup Tree_E(n) \\ & \cup \left \{(c^2_i, d_{4i-3}^8),(c^2_i, d_{4i-2}^8),(c^2_i, d_{4i-1}^8),(c^2_i, d_{4i}^8) \mid 1\leq i \leq m\right\} \\ & \cup \left \{(d^4_{2i-1}, d_{4i-3}^8),(d^4_{2i-1}, d_{4i-2}^8),(d^4_{2i-1}, d_{4i-1}^8),(d^4_{2i-1}, d_{4i}^8) \mid 1\leq i \leq m\right\} \\ & \cup \left \{(d^4_{2i}, d_{4i-3}^8),(d^4_{2i}, d_{4i-2}^8),(d^4_{2i}, d_{4i-1}^8),(d^4_{2i}, d_{4i}^8) \mid 1\leq i \leq m\right\}  \\ & \cup \left \{(c^2_i, \sigma^2_{3i-2}),(c^2_j, \sigma^2_{3i-1}),(c^2_i, \sigma^2_{3i}) \mid 1 \leq i \leq m\right\} \\ & \cup \left\{(d^6_i, \sigma^2_{3i - 2}), (d^6_i, \sigma^2_{3i-1}),(d^6_i, \sigma^2_{3i}) \mid 1\leq i\leq m\right\} \\ & \cup \left\{(d^7_i, \sigma^1_{3i-2}),(d^7_i, \sigma^1_{3i-1}),(d^7_i, \sigma^1_{3i}) \mid 1\leq i\leq m\right\}\\& \cup \left\{(\chi_i, c^2_{i+1}) \mid 1\leq i\leq m-1\right\}\\& \cup \left\{(\chi_i, \sigma^1_{3i-2}), (\chi_i, \sigma^1_{3i-1}), (\chi_i, \sigma^1_{3i}) \mid 1\leq i\leq m \right\} \\ &  \cup \left\{(\chi_i, \sigma^2_{3i-2}), (\chi_i, \sigma^2_{3i-1}), (\chi_i, \sigma^2_{3i}) \mid 1\leq i\leq m \right\} \\ &  \cup \left\{(\chi_i, d^{13}_i) \mid 1\leq i\leq m\right\} \cup \left\{(d^{13}_i, d^{14}_{2i-1}), (d^{13}_i, d^{14}_{2i}) \mid 1\leq i\leq m\right\}\\ & \cup \left\{(\chi_m, EndGame)\right\}
    \\ \\
    \textcolor{blue}{B(G)} = &\left\{(c_1^1,d_1^1),(c_1^1,d_2^1),(c_1^1,d_3^1) \right\} \\ & \cup \left\{(c_j^1,v_j^0),(c_j^1,v_j^1),(c_j^1,d_{j}^5), (v_j^0, v_j^1) \mid 1 \leq j \leq n\right\} \\ &\cup \left\{(v_j^0, c_{j+1}^1),(v_j^1, c_{j+1}^1) \mid 1 \leq j \leq n-1\right\}  \\ & \cup \left\{ (v_j^0, d^3_{\left| \c{C}^{*}(j)\right| + i}) \Big\vert 1\leq i \leq \left| \c{C}^0(j+1)\right| - \left| \c{C}^0(j)\right|\right\}  \\ & \cup \left\{ (v_j^1, d^3_{\left| \c{C}^{*}(j)\right| +\left| \c{C}^0(j+1)\right| - \left| \c{C}^0(j)\right| + i}) \Big\vert 1\leq i \leq \left| \c{C}^1(j+1)\right| - \left| \c{C}^1(j)\right|\right\} \\ & \cup \left\{(d_i^3,d_{2i - 1}^2), (d_i^3,d_{2i}^2) \mid 1 \leq i \leq 3m\right\}
    \\ \\
    \textcolor{violet}{P(G)} = &\left\{(v_j^0, \sigma^1_{y_1(k)}) \mid k\in y_2(\beta_j)\cap\c{C}^0 \right\} \cup \left\{(v_j^1, \sigma^1_{y_1(k)}) \mid k\in y_2(\beta_j)\cap \c{C}^1 \right\}
\end{align*} In the above, $Tree_E(n) = \{(\lambda_2,c_1^2)\}$ if $n$ is even, whereas if $n$ is odd we have $Tree_E(n) = \{(\lambda_2, \lambda_3), (\lambda_3, d_1^{9}), (\lambda_3, d_2^{9}), (d_1^{9}, d_1^{10}), (d_1^{9}, d_2^{10}), (d_2^{9}, d_3^{10}), (d_2^{9}, d_4^{10}), (\lambda_3, c_1^2)\}$.



Figure \ref{fig:log-complete ref} shows a simple example of the Toggle position associated with the $3$-$QBF$ input $A=\big(\widehat \beta, \varphi\big)$, where $\widehat \beta=\{\beta_1, \beta_2, \beta_3 \}$ and $\varphi = \beta_1 \vee \beta_2 \vee \beta_3$. Figure $\ref{fig:qbfTogExam}$ illustrates a more robust example of the Toggle position associated with input $A = \big( \widehat \beta,\varphi\big)$ where $\widehat \beta = \{\beta_1, \beta_2, \beta_3,\beta_4,\beta_5,\beta_6\}$ and $\varphi = (\beta_1 \vee \neg\beta_2 \vee \beta_3) \wedge (\neg \beta_1 \vee \beta_4 \vee \beta_5)\wedge (\beta_1 \vee \neg\beta_5 \vee \beta_6) \wedge (\neg\beta_3 \vee \beta_6 \vee \neg\beta_6)$. This visual is especially helpful for understanding how the game extends and, along with Figure \ref{fig:log-complete ref}, for contrasting the cases where the number of variables is odd or even.

\begin{lemma}
    For each instance of a $3$-$QBF$ game $G_\omega(3\text{-}QBF)$, there exists a Toggle game logically equivalent to $G_\omega(3\text{-}QBF)$, i.e.\ a winning strategy in the Toggle game implies a corresponding winning strategy in $G_\omega(3\text{-}QBF)$ and vice versa.
\end{lemma}

\begin{proof}
    
Let $A=\big(\widehat \beta,\varphi\big)$ be the input for an instance of $G_\omega(3\text{-}QBF)$, and let $\{G,V_0^{(0)}\}$ be the Toggle position corresponding to input $A$. The Toggle game begins by playing first on the variable vertices. 
At the outset, only two possible Toggle moves are available to Player $1$, i.e.\ vertices $v_1^0$ and $v_1^1$. If Player $1$ plays on $v_1^0$, the variable $\beta_1$ is assigned as False. Similarly, if Player $1$ plays on $v_1^1$, then $\beta_1$ is assigned True. This means the only moves available to Player $2$ are $v_2^0$ and $v_2^1$, and a truth value for $\beta_{2}$ is assigned in the same manner (see Figure \ref{thisisalabel}). The players alternate assigning variables until each of $\beta_1,\dots,\beta_n$ has been assigned a truth value. 
In accordance with Figures \ref{fig:log-complete ref}-\ref{thisisalabel}, the blue section (vertices and edges) will no longer be played. The balance of the game will be played on the red section.

After all variables have been assigned truth values, only one move is available to the next player. If $n$ is even, then Player $1$ must play on the link vertex $\lambda_2$. If $n$ is odd, then Player $2$ must play on $\lambda_2$ and Player $1$ must follow by playing on $\lambda_3$. This is a consequence of the fact that Player 2 is required to play on the control vertices $c_i^2$ and Player $1$ is required to play on the clause vertices $\chi_i$. 

The Toggle game now enters its second phase,
wherein each player will have at most one playable vertex per turn until the game concludes. After all link vertices $\lambda_i$ have been played, Player $2$ must play on $c_1^2$. At this point, if at least one of the signal vertices $\sigma^1_{1}$, $\sigma^1_{2}$, $\sigma^1_{3}$ has weight $1$, then $\chi_1$ is playable. Notice that this occurs only if either at least one variable $\beta_i$ with $(\beta_i,\chi_1) \in \c{C}^0$ has been assigned False or at least one variable $\beta_i$ with $(\beta_i,\chi_1) \in \c{C}^1$ has been assigned True. Thus vertex $\chi_1$ becomes playable if and only if clause $\chi_1$ is True. In this case Player $1$ must now play on $\chi_1$, rendering $c_2^2$ playable (see Figure \ref{ClauseSatCertValueTable}). This process repeats until the game ends, either when some $\chi_i$ is not playable, in which case Player $2$ wins, or all $\chi_i$ and $c_i^2$ vertices have been played, in which case Player 1 wins. Note that Player $1$ wins the Toggle game if and only if all clauses $\chi_i$ are assigned the value True, which makes this game logically equivalent to the $3$-$QBF$ game, i.e.\ a winning strategy in this Toggle game implies a winning strategy in the $3$-$QBF$ game, and conversely.
\end{proof}




\begin{theorem}\label{thm:Toggle_Pspace}
    There exists a logspace reduction from the PSPACE-complete problem $3$-$QBF$ to Toggle. Thus the game of Toggle is PSPACE-complete.
\end{theorem}
\begin{proof}    
We now determine the space complexity of Toggle relative to that of $3$-$QBF$.

Let $A=\big(\widehat \beta,\varphi\big)$ and $A'=\big(\widehat \beta \cup \{\beta_{n+1}\},\varphi\big)$ be inputs for instances of a $3$-$QBF$ game. Let $G$ and $G^{\,\prime}$ be the graphs of the Toggle games associated with $A$ and $A'$, respectively. Define $Tree_V(n)=\emptyset$ if $n$ is even and $Tree_V(n)=\{\lambda_3,d_1^9,d_2^9,d_1^{10},d_2^{10},d_3^{10},d_4^{10}\}$ if $n$ is odd. Then the vertex set and edge set of $G^{\,\prime}$ are given as follows:\begin{align*}V(G^{\,\prime}) = &\big(V(G) \setminus Tree_V(n)\big) \cup Tree_V(n+1) \cup \left\{c^1_{n+1},d^5_{n+1}, v^0_{n+1}, v^1_{n+1}\right\}
    \\E(G^{\,\prime}) = & \big(E(G) \setminus \big(Tree_E(n) \cup \left\{(v_n^0,\lambda_1), (v_n^1, \lambda_1)\right\}\big)\big) \cup Tree_E(n+1) \\ &\cup \{(v_{n}^0,c_{n+1}^1),(v_{n}^1,c_{n+1}^1),(v_{n+1}^0, \lambda_1), (v_{n+1}^1, \lambda_1)\} \\& \cup \left\{(c_{n+1}^1,v_{n+1}^0),(c_{n+1}^1,v_{n+1}^1),(c_{n+1}^1,d_{n+1}^5), (v_{n+1}^0, v_{n+1}^1)\right\}
\end{align*}

This indicates that the marginal space complexity for an additional $CNF$ variable $\beta_{n+1}$ is bounded above by a constant, i.e.\ not a function of $m$ nor $n$, see Figure \ref{fig:PSpaceMargVar}.

Now let $A=\big(\widehat \beta,\varphi\big)$ be the base instance of a $3$-$QBF$ game and consider the marginal space complexity of adding a clause $\chi_{n+1}=\beta_{k_1} \vee \beta_{k_2} \vee \beta_{k_3}$ to $A$, where $k_1, k_2, k_3 \in \{1, ...,m\}$. Then $A''=\big(\widehat \beta, \varphi \wedge \chi_{n+1}\big)$ is a valid instance of a $3$-$QBF$ game. Let $G$ and $G^{\,\prime\prime}$ be the graphs of the Toggle games associated with $A$ and $A''$, respectively. Then the vertex set and edge set of $G^{\,\prime\prime}$ are given as follows:\begin{align*}
     V(G^{\,\prime\prime}) = & V(G) \cup \left\{d^3_i, \sigma^1_i, \sigma^2_i \mid 3m+1 \leq i \leq 3m + 3\right\}\\ & \cup \left\{c^2_{m+1}, \chi_{m+1}, d^6_{m+1}, d^{13}_{m+1}, d^7_{m+1}\right\} \\& \cup \left\{d^8_1 \mid 4m + 1 \leq i \leq 4m + 4\right\}\\ & \cup \left\{d^4_i, d^{14}_i \mid 2m + 1 \leq i \leq 2m + 2\right\} \\& \cup \left\{d^2_i \mid 6m + 1 \leq i \leq 6m + 6\right \}
     \\ \\ 
     E(G^{\,\prime\prime}) = & (E(G) \setminus \{(\chi_m, EndGame), (\chi_m, c^2_{m+1})\})\\ & \cup  \left\{ (v_{k_i}^1, d^3_{3m + i}), (v_{k_i}^1, \sigma_{3m + i}^1), (d_{3m+i}^3,d_{3m + 2i - 1}^2), (d_{3m + i}^3,d_{3m + 2i}^2) \mid 1\leq i \leq 3\right\} \\ & \cup \left\{(c^2_{m+1}, d_{4m+1}^8),(c^2_{m+1}, d_{4m+2}^8),(c^2_{m+1}, d_{4m+3}^8),(c^2_{m+1}, d_{4m+4}^8) \right\} \\ & \cup \left\{(d^4_{2m+1}, d_{4m+1}^8),(d^4_{2m+1}, d_{4m+2}^8),(d^4_{2m+1}, d_{4m+3}^8),(d^4_{2m+1}, d_{4m+4}^8)\right\} \\ & \cup \left\{(d^4_{2m+2}, d_{4m+1}^8),(d^4_{2m+2}, d_{4m+2}^8),(d^4_{2m+2}, d_{4m+3}^8),(d^4_{2m+2}, d_{4m+4}^8)\right\}  \\ & \cup \left \{(c^2_{m+1}, \sigma^2_{3m+1}),(c^2_{m+1}, \sigma^2_{3m+2}), (c^2_{m+1}, \sigma^2_{3m+3})\right\} \\&\cup \left\{(d^6_{m+1}, \sigma^2_{3m + 1}), (d^6_{m+1}, \sigma^2_{3m + 2}), (d^6_{m+1}, \sigma^2_{3m + 3})\right\} \\  & \cup \left\{(d^7_{m+1}, \sigma^1_{3m + 1}), (d^7_{m+1}, \sigma^1_{3m + 2}), (d^7_{m+1}, \sigma^1_{3m + 3})\right\}  \\ &  \cup \left\{(\chi_{m+1}, \sigma^1_{3m+1}), (\chi_{m+1}, \sigma^1_{3m+2}), (\chi_{m+1}, \sigma^1_{3m+2}) \right\}\\ &  \cup \left\{(\chi_{m+1}, \sigma^2_{3m+1}), (\chi_{m+1}, \sigma^2_{3m+2}), (\chi_{m+1}, \sigma^2_{3m+2}) \right\} \\ &  \cup \left\{(\chi_{m+1}, d^{13}_{m+1})\right\} \cup \left\{(d^{13}_{m+1}, d^{14}_{2m+1}), (d^{13}_{m+1}, d^{14}_{2m+2})\right\}\\ & \cup \left\{(\chi_{m+1}, EndGame)\right\}
\end{align*}

The above, along with the visual illustration in Figure \ref{fig:PSpaceMargCla}, demonstrates that the marginal space complexity for an additional $CNF$ clause $\chi_{n+1}$ is a uniform constant independent of $m$ and $n$. Formally, the inclusion of an additional clause increases the size of the associated Toggle graph by a net gain of 28 vertices and 43 edges.

Given a $3$-$QBF$ instance, the marginal space complexity in Toggle of including another variable or clause is bounded above by a constant, as previously shown. Thus, if we assume that the spatial complexity of $3$-$QBF$ is $f(|A|)$, where $|A|$ is an input length indicator, then the spatial complexity of the Toggle graph is $c_1\cdot f(|A|)$. Therefore, since the latter spatial complexity is proportional to that of $3$-$QBF$ it follows that Toggle is PSPACE-complete.
\end{proof}

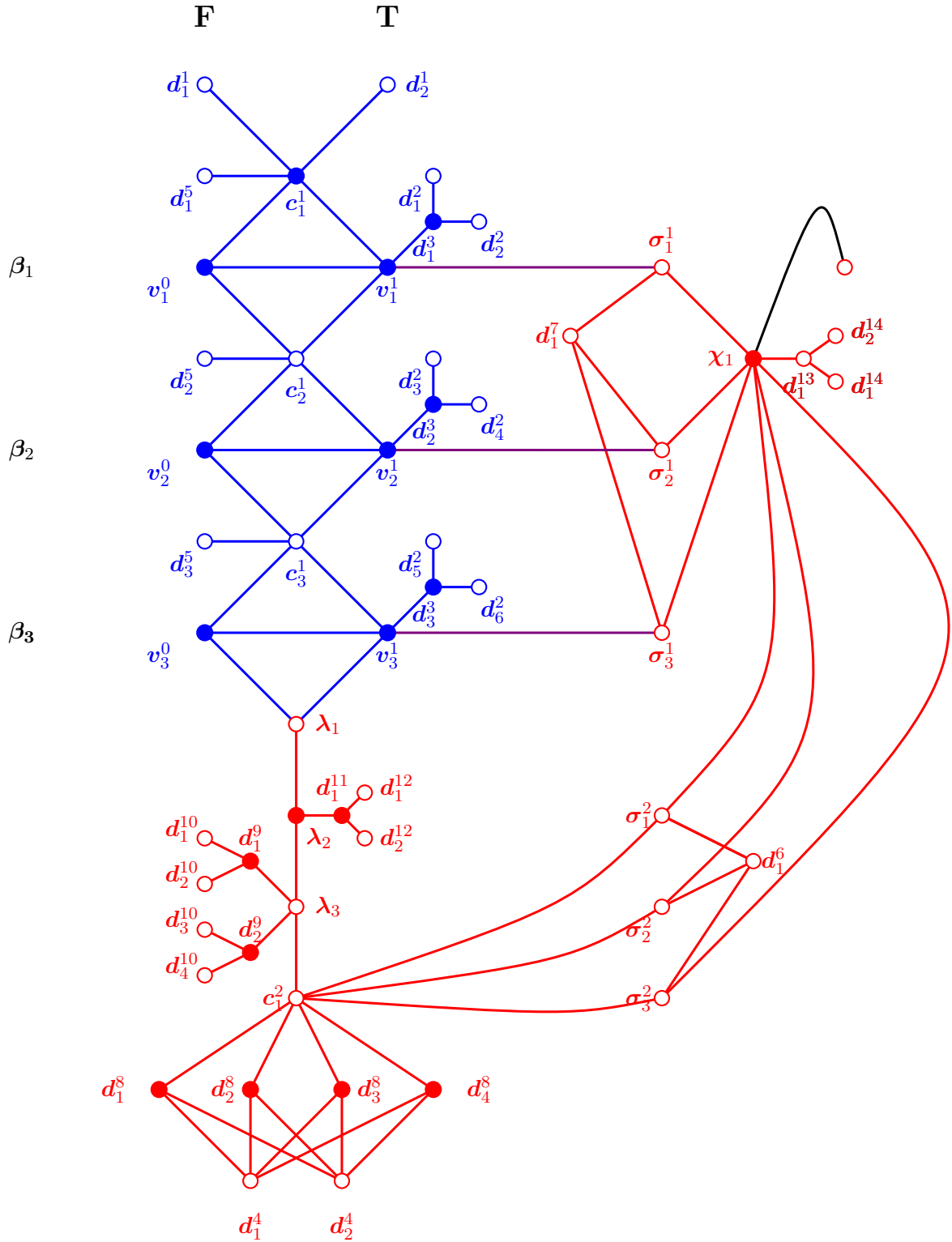
\begin{figure}[H]
        \begin{center}
            
        \begin{tikzpicture}[xscale=-1, rotate=270, scale = 1.5, font=\small, line width = 0.4mm]
            \coordinate(a1) at (-3,3);
            \coordinate(a2) at (-1,3);
            \coordinate(a3) at (0,3);
            \coordinate(a4) at (1,3);
            \coordinate(a5) at (2,3);
            \coordinate(a6) at (3,3);
            \coordinate(b1) at (-2,3);
            \coordinate(b2) at (-2,2);
            \coordinate(b3) at (0,2);
            \coordinate(b4) at (2,2);
            \coordinate(b5) at (7,2);
            \coordinate(c1) at (-3,1);
            \coordinate(c2) at (-1,1);
            \coordinate(c3) at (1,1);
            \coordinate(c4) at (3,1);
            \coordinate(d2) at (-1,-2);
            \coordinate(d3) at (1,-2);
            \coordinate(d4) at (3,-2);
            \coordinate(d6) at (5,-2);
            \coordinate(d7) at (6,-2);
            \coordinate(d8) at (7,-2);
            \coordinate(e3) at (0,-3);
            \coordinate(f2) at (-1,-4);
            \coordinate(g1) at (-2,0.5);
            \coordinate(g2) at (-1.5,0);

            \coordinate(g3) at (0,0.5);
            \coordinate(g4) at (0.5,0);

            \coordinate(g5) at (2,0.5);
            \coordinate(g6) at (2.5,0);

            \coordinate(n1) at (5,2);
            \coordinate(s14) at (5,1.5);
            \coordinate(s15) at (4.75, 1.25);
            \coordinate(s16) at (5.25, 1.25);
            \coordinate(n2) at (6,2);
            \coordinate(n3) at (4,2);
            \coordinate(s8) at (5.5,2.5);
            \coordinate(s9) at (6.5,2.5);
            
            \coordinate(s10) at (5.25,3);
            \coordinate(s11) at (5.75,3);
            \coordinate(s12) at (6.25,3);
            \coordinate(s13) at (6.75,3);

            \coordinate(w1) at (8, 3.5);
            \coordinate(w2) at (8, 2.5);
            \coordinate(w3) at (8, 1.5);
            \coordinate(w4) at (8, 0.5);
            \coordinate(s1) at (9, 2.5);
            \coordinate(s2) at (9, 1.5);
            \coordinate(s3) at (5.5,-3);
            \coordinate(s4) at (-1.5,0.5);
            \coordinate(s5) at (0.5,0.5);
            \coordinate(s6) at (2.5,0.5);
            \coordinate(s7) at (-0.25, -1);

            \coordinate(s17) at (0,-3.55);
            \coordinate(s18) at (0.25,-3.9);
            \coordinate(s19) at (-0.25,-3.9);
            \node at (0.3,-3.5){${\bm d^{13}_1}$};
            \node at (0.3,-4.25){${\bm d^{14}_1}$};
            \node at (-0.3,-4.25){${\bm d^{14}_2}$};

            \draw[red](d8) .. controls (3,-6) ..(e3);
            \draw[red](b5) .. controls (7.2,-1) ..(d8);
            \draw[red](d8)--(s3);
            
            \draw[red](d2)--(s7)--(d3);
            \draw[red](d4)--(s7);

            \draw[blue](a1)--(b2)--(c1);
            \draw[blue](b1)--(b2)--(a2)--(b3)--(a4)--(b4)--(a6)--(n3)--(c4)--(b4)--(c3)--(b3)--(c2)--(b2);
            \draw[blue](a2)--(c2);
            \draw[violet](c2)--(d2);
            \draw[red](d2)--(e3)--(d3);
            \draw[blue](c3)--(a4);
            \draw[red](e3)--(d4);
            \draw[blue](c4)--(a6);
            \draw[red](d6)--(s3)--(d7);
            \draw(f2) .. controls (-2, -3.75) .. (e3);
            \draw[blue](b3)--(a3);
            \draw[blue](b4)--(a5);
            \draw[blue](c2)--(s4);
            \draw[blue](g1)--(s4)--(g2);

            \draw[blue](c3)--(s5);
            \draw[blue](g3)--(s5)--(g4);

            \draw[blue](c4)--(s6);
            \draw[blue](g5)--(s6)--(g6);
            \draw[red](b5)--(n2)--(n1)--(n3);
            \draw[red](s3)--(d6);
            \draw[red](w1)--(b5)--(w2);
            \draw[red](w3)--(b5)--(w4);
            \draw[red](w1)--(s1)--(w2)--(s2)--(w3)--(s1)--(w4)--(s2)--(w1);
            \draw[red](b5) .. controls (6.6,-1) ..(d7);
            \draw[red](b5) .. controls (6,-1) ..(d6);
            \draw[red](d7) .. controls (4.1,-4) ..(e3);
            \draw[red](d6) .. controls (3.5,-3.4) ..(e3);
            \draw[violet](d4) -- (c4);
            \draw[violet](d3) .. controls (1,0) ..(c3);
            
            \draw[red](n2)--(s8)--(s10);
            \draw[red](s8)--(s11);

            \draw[red](n2)--(s9)--(s12);
            \draw[red](s9)--(s13);
            \draw[red](n1)--(s14);
            
            \draw[red](s14)--(s15);
            \draw[red](s14)--(s16);
            \draw[red](s17)--(s18);
            \draw[red](s17)--(s19);
            \draw[red](s17)--(e3);
            \foreach\i in{a2,a4,a6,b2,c2,c3,c4,s4,s5,s6}{\filldraw(\i)[color=blue]circle(0.08);}
            \foreach\i in{w1,w2,w3,w4,e3, n1, s8, s9, s14}{\filldraw(\i)[color=red]circle(0.08);}

            \foreach\i in{a1,a3,a5,b1,b3,b4,c1,g1,g2,g3,g4,g5,g6}{\filldraw(\i)[color=blue,fill=white!90,thick]circle(0.08);}

            \foreach\i in{n2,n3,d2,d3,d4,d6,d7,d8,b5,s1,s2,s3,s7,s10, s11, s12, s13, s15, s16, s17, s18, s19,f2}{\filldraw(\i)[color=red,fill=white!90,thick]circle(0.08);}            

            \node at (-3.75,3){\textbf{\large F}};
            \node at (-3.75,1){\textbf{\large T}};
            \node at (-1,5){${\bm\beta_1}$};
            \node at (1,5){${\bm\beta_2}$};
            \node at (3,5){${\bm{\beta_3}}$};

            \node[blue] at (-3,3.25){${\bm d_1^1}$\;};
            \node[blue] at (-3,0.6){${\bm d_2^1}$\;\;};

            \node[red] at (-1.3,-2){${\bm \sigma^1_{1}}$};
            \node[red] at (1.25,-2){${\bm \sigma^1_{2}}$};
            \node[red] at (3.25,-2){${\bm \sigma^1_{3}}$};
            \node[red] at (-0.25,-0.75){${\bm d^7_1}$};

            \node[red] at (0,-2.65){${\bm \chi_1}$};

            \node[red] at (8,4){${\bm d^8_1}$};
            \node[red] at (8,2.8){${\bm d^8_2}$};
            \node[red] at (8,1.2){${\bm d^8_3}$};
            \node[red] at (8,0){${\bm d^8_4}$};
            \node[red] at (9.5,2.5){${\bm d^4_1}$};
            \node[red] at (9.5,1.5){${\bm d^4_2}$};

            \node[blue] at (-0.75,3.5){${\bm v_1^0}$};
            \node[blue] at (1.25,3.5){${\bm v_2^0}$};
            \node[blue] at (3.25,3.5){${\bm v_3^0}$};
            
            \node[blue] at (-1.75,3.25){${\bm d_1^5}$};            
            \node[blue] at (0.25,3.25){${\bm d_2^5}$};
            \node[blue] at (2.25,3.25){${\bm d_3^5}$};

            \node[blue] at (-0.75,1){${\bm v_1^1}$};
            \node[blue] at (1.25,1){${\bm v_2^1}$};
            \node[blue] at (3.25,1){${\bm v_3^1}$};

            \node[blue] at (-1.7,2){${\bm c_1^1}$};
            \node[blue] at (0.35,2){${\bm c_2^1}$};
            \node[blue] at (2.35,2){${\bm c_3^1}$};
            \node[red] at (7,2.25){${\bm c_1^2}$};

            \node[red] at (4,1.65){${\bm \lambda_1}$};
            \node[red] at (5.25,1.75){${\bm \lambda_2}$};
            \node[red] at (6,1.65){${\bm \lambda_3}$};

            \node[red] at (4.7, 1.6){${\bm d^{11}_1}$};
            \node[red] at (4.7, 0.9){${\bm d^{12}_1}$};
            \node[red] at (5.25, 0.9){${\bm d^{12}_2}$};

            \node[red] at (5.25, 2.5){${\bm d^9_1}$};
            \node[red] at (6.25, 2.5){${\bm d^9_2}$};
            
            \node[red] at (5.15, 3.25){${\bm d^{10}_1}$};
            \node[red] at (5.65, 3.25){${\bm d^{10}_2}$};

            \node[red] at (6.15, 3.25){${\bm d^{10}_3}$};
            \node[red] at (6.65, 3.25){${\bm d^{10}_4}$};

            \node[blue] at (-1.2,0.6){${\bm d_1^3}$};
            \node[blue] at (0.8,0.6){${\bm d_2^3}$};
            \node[blue] at (2.8,0.6){${\bm d_3^3}$};

            \node[red] at (7,-1.75){${\bm \sigma^2_3}$};            
            \node[red] at (6.25,-1.75){${\bm \sigma^2_2}$};
            \node[red] at (5,-1.75){${\bm \sigma^2_1}$};
            \node[red] at (5.5,-3.22){${\bm d^6_1}$};

            \node[blue] at (-1.75,0.75){${\bm d_1^2}$};
            \node[blue] at (0.25,0.75){${\bm d_3^2}$};
            \node[blue] at (2.25,0.75){${\bm d_5^2}$};
            \node[blue] at (-1.25,-0.15){${\bm d_2^2}$};
            \node[blue] at (0.75,-0.15){${\bm d_4^2}$};
            \node[blue] at (2.75,-0.15){${\bm d_6^2}$};
            
            \node[red] at (0.3,-3.5){${\bm d^{13}_1}$};
            \node[red] at (0.3,-4.25){${\bm d^{14}_1}$};
            \node[red] at (-0.3,-4.25){${\bm d^{14}_2}$};


        \end{tikzpicture}\vspace*{-2mm}
        \caption{An instance of Toggle that is logically equivalent to the $3$-$QBF$ game $(\beta_1 \vee \beta_2 \vee \beta_3)$. The solid vertices have weight $1$ and the empty vertices have weight $0$. (Colors are consistent with the notation of the edge sets $\textcolor{red}{R(G)},\textcolor{blue}{B(G)},\textcolor{violet}{P(G)}$.) }
        \label{fig:log-complete ref}
        \end{center}
    \end{figure}

\begin{figure}[H]
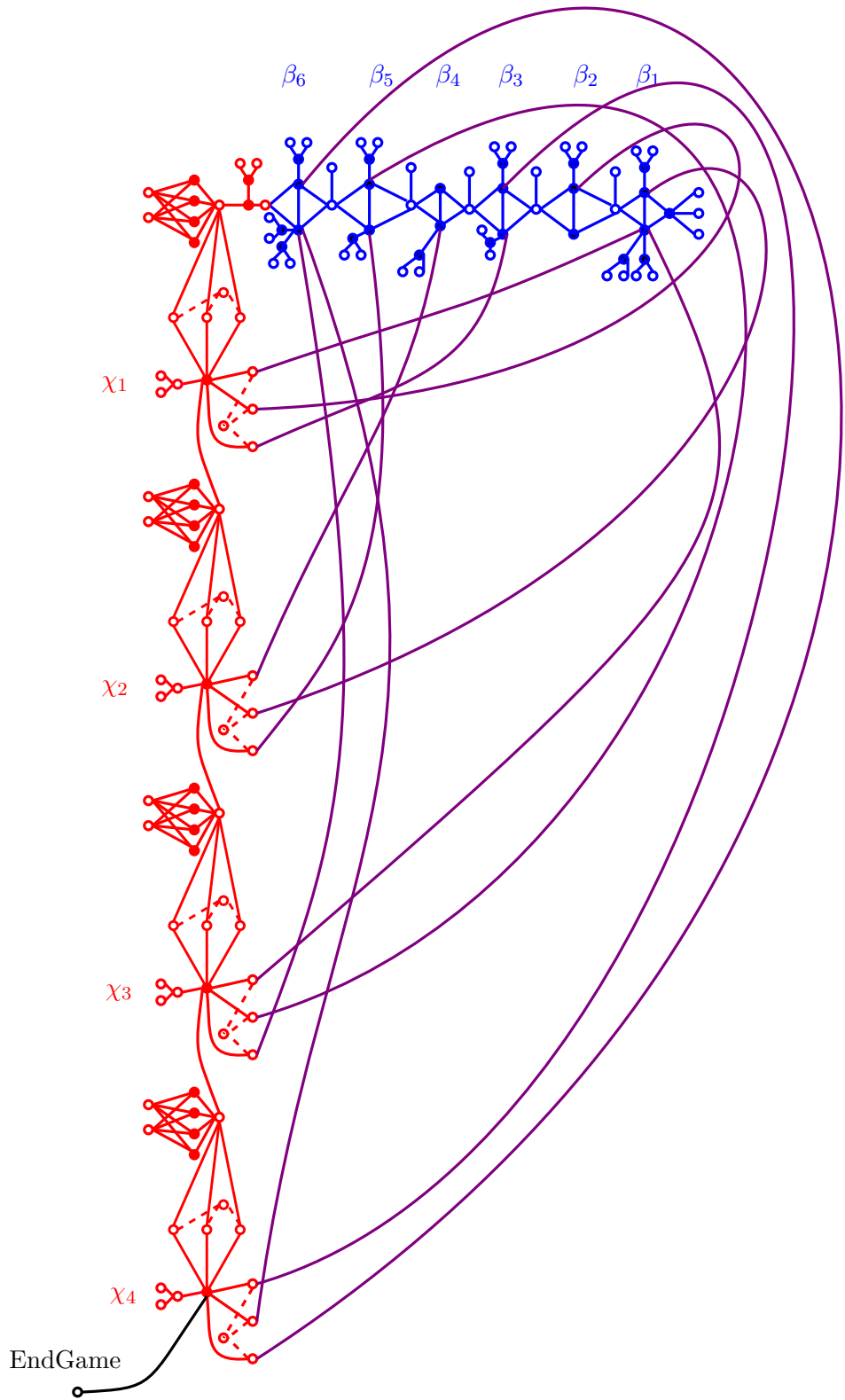

    \vspace*{-8cm}
    \begin{center}

    \caption{An instance of Toggle that is logically equivalent to the $3$-$QBF$ game \textbf{($\beta_1$ $\vee$ $\neg$$\beta_2$ $\vee$ $\beta_3$) $\wedge$ ($\neg$ $\beta_1$ $\vee$ $\beta_4$ $\vee$ $\beta_5$) $\wedge$ ($\beta_1$ $\vee$ $\neg$$\beta_5$ $\vee$ $\beta_6$) $\wedge$ ($\neg$$\beta_3$ $\vee$ $\beta_6$ $\vee$ $\neg$$\beta_6$)}. The solid vertices have weight $1$ and the empty vertices have weight $0$. (Colors are consistent with the notation of the edge sets $\textcolor{red}{R(G)},\textcolor{blue}{B(G)},\textcolor{violet}{P(G)}$. Note that the dashed lines are equivalent to solid lines and only differ for visual clarity.)}
    \label{fig:qbfTogExam}
    \end{center}
\end{figure}

\begin{figure}[H]
\centering
\begin{minipage}{.5\textwidth}
  \centering
  \subfigure{
\vspace*{-2mm}}
\text{(i): $(\beta_1 \vee \beta_2 \vee \beta_3)$ after one move at $v^1_1$.}
\vspace{1cm}
\end{minipage}%
\begin{minipage}{.5\textwidth}
  \centering
  \subfigure{
        \vspace{0.5mm}
        %
\vspace*{-2mm}}
\text{(ii): After one additional move at $v^0_2$.}
\vspace{1cm}
\end{minipage}
\resizebox{1\textwidth}{!}{%
\begin{tabular}{|c|c||c|c|c||c|c|c||c|c|c|}
                \hline
             & & \multicolumn{3}{c||}{Initial Position} & \multicolumn{3}{c||}{After One Move} & \multicolumn{3}{c|}{After Two Moves} \\                
             \hline\rule{0pt}{2.5ex}\rule[-1.2ex]{0pt}{0pt}
             $v$ & $|N[v]|$ & $\omega^{(0)}(v)$ &$\sigma^{(0)}(v)$ & Playable &$\omega^{(1)}(v)$ &$\sigma^{(1)}(v)$ & Playable & $\omega^{(2)}(v)$ &$\sigma^{(2)}(v)$ & Playable  \\
             \hline\rule{0pt}{2.5ex}\rule[-1.2ex]{0pt}{0pt}
             \textcolor{blue}{$c^1_1$} & 6 & 3 & 1 & No & 0 & 0 & No & 0 & 0 & No\\ 
                  \hline\rule{0pt}{2.5ex}\rule[-1.2ex]{0pt}{0pt}
        
             \textcolor{blue}{$v^0_1$} & 4 & 3 & 1 & \textbf{Yes} & 1 & 0 & No & 0 & 0 & No\\ 
                  \hline\rule{0pt}{2.5ex}\rule[-1.2ex]{0pt}{0pt}
        
             \textcolor{blue}{$v^1_1$} & 6 & 4 & 1 & \textbf{Yes} & 2 & 0 & No & 1 & 0 & No\\ 
                  \hline\rule{0pt}{2.5ex}\rule[-1.2ex]{0pt}{0pt}
        
             \textcolor{blue}{$c^1_2$} & 6 & 4 & 0 & No & 3 & 1 & No & 0 & 0 & No\\ 
                  \hline\rule{0pt}{2.5ex}\rule[-1.2ex]{0pt}{0pt}
        
             \textcolor{blue}{$v^0_2$} & 4 & 2 & 1 & No & 3 & 1 & \textbf{Yes} & 1 & 0 & No\\ 
                  \hline\rule{0pt}{2.5ex}\rule[-1.2ex]{0pt}{0pt}
        
             \textcolor{blue}{$v^1_2$} & 6 & 3 & 1 & No & 4 & 1 & \textbf{Yes} & 2 & 0 & No\\ 
                  \hline\rule{0pt}{2.5ex}\rule[-1.2ex]{0pt}{0pt}
        
             \textcolor{blue}{$c^1_3$} & 6 & 4 & 0 & No & 4 & 0 & No & 3 & 1 & No\\ 
                  \hline\rule{0pt}{2.5ex}\rule[-1.2ex]{0pt}{0pt}
        
             \textcolor{blue}{$v^0_3$} & 4 & 2  & 1 & No & 2 & 1 & No & 3 & 1 & \textbf{Yes}\\ 
                  \hline\rule{0pt}{2.5ex}\rule[-1.2ex]{0pt}{0pt}
        
             \textcolor{blue}{$v^1_3$} & 6 & 3 & 1 & No & 3 & 1 & No & 4 & 1 & \textbf{Yes}\\
                  \hline
        
        \end{tabular}}%

        
        
        
        
        
        
        
        
        
\caption{The value of vertex neighborhoods after one and two moves from the initial position in Figure \ref{fig:log-complete ref}. (Enlarged circles denote the toggled vertex at that specific turn.)}

\label{thisisalabel}
\end{figure}

\begin{figure}[H]
\centering
\hspace{-4cm}
\begin{minipage}{.25\textwidth}
  \centering
  \subfigure{
%
}\vspace{4mm}
\text{(i) Beginning of certifying clause $i$.}
\vspace{1cm}
  \subfigure{      
        %
}\vspace{-4mm}
\text{(iii) Second move made by Player 1 at vertex $\chi_i$.}
\vspace{2mm}
\end{minipage}
\hspace{5cm}
\begin{minipage}{.25\textwidth}
  \centering
  \subfigure{
            %
}\vspace{4mm}
\text{(ii) First move made by Player 2 at vertex $c_i^2$.}
\vspace{1cm}
  \subfigure{
    %
}\vspace{-4mm}
\text{(iv) Third move made by Player 2 at vertex $c_{i+1}^2$.}
\vspace{2mm}
\end{minipage}
\resizebox{1\textwidth}{!}{%
\begin{tabular}{|c|c||c|c|c||c|c|c||c|c|c||c|c|c|}
                \hline
             & & \multicolumn{3}{c||}{At stage $j$} & \multicolumn{3}{c||}{After One Move} & \multicolumn{3}{c|}{After Two Moves} & \multicolumn{3}{c|}{After Three Moves} \\                
             \hline\rule{0pt}{2.5ex}\rule[-1.2ex]{0pt}{0pt}
             $v$ & $|N[v]|$ & $\omega^{(j)}(v)$ &$\sigma^{(j)}(v)$ & Playable &$\omega^{(j+1)}(v)$ &$\sigma^{(j+1)}(v)$ & Playable & $\omega^{(j+2)}(v)$ &$\sigma^{(j+2)}(v)$ & Playable & $\omega^{(j+3)}(v)$ &$\sigma^{(j+3)}(v)$ & Playable  \\
             \hline\rule{0pt}{2.5ex}\rule[-1.2ex]{0pt}{0pt}
             \textcolor{red}{$c^2_i$} & 9 & 5 & 1 & \textbf{Yes} & 4 & 0 & No & 1 & 0 & No & 1 & 0 & No\\ 
                  \hline\rule{0pt}{2.5ex}\rule[-1.2ex]{0pt}{0pt}
        
             \textcolor{red}{$\chi_i$} & 9 & 2 & 1 & No & 5 & 1 & \textbf{Yes} & 4 & 0 & No & 4 & 1 & No\\ 
                  \hline\rule{0pt}{2.5ex}\rule[-1.2ex]{0pt}{0pt}
        
             \textcolor{red}{$c^2_{i+1}$} & 9 & 5 & 0 & No & 5 & 0 & No & 5 & 1 & \textbf{Yes} & 4 & 0 & No\\ 
                  \hline
        
        \end{tabular}}%
\caption{The process of certifying that clause $\chi_i$ is True based on which variable vertices were previously played. The value of vertex neighborhoods at beginning and after one, two, and three moves. (Enlarged circles denote the toggled vertex at that specific turn.)}
\label{ClauseSatCertValueTable}
\end{figure}

\begin{figure}[H]
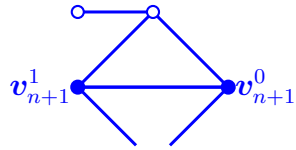

    \begin{center}
    %
\vspace*{-2mm}
    \caption{Marginal Toggle graph component for each additional variable in the $CNF$ formula.}
    \label{fig:PSpaceMargVar}
    \end{center}
\end{figure}

\begin{figure}[H]
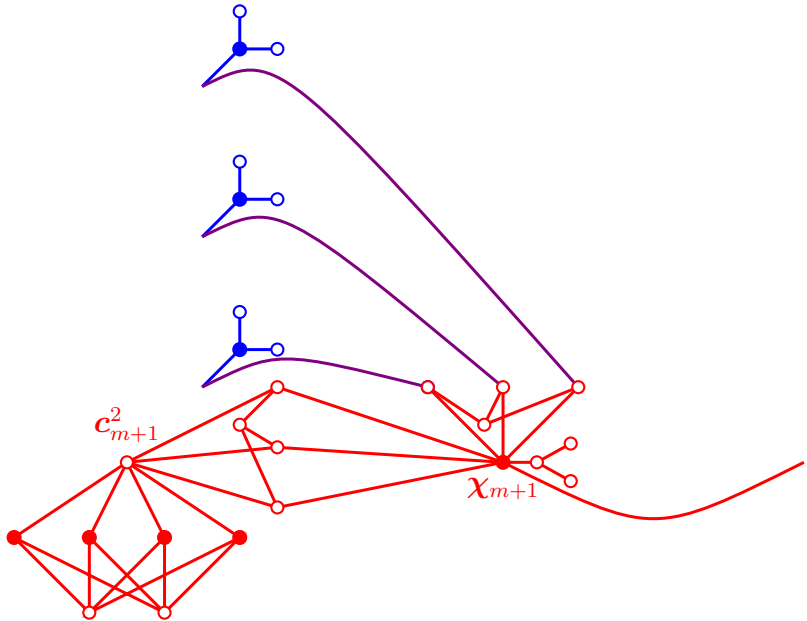

    \begin{center}
    %
\vspace*{2mm}
    \caption{Marginal Toggle graph component for each additional clause in the $CNF$ formula.}
    \label{fig:PSpaceMargCla}
    \end{center}
\end{figure}

\begin{remark}
     Theorem \ref{thm:Toggle_Pspace} was expected and alluded to by Even and Tarjan \cite{ET75} who posited that ``any game with a sufficiently rich structure" would (according to current theory) be PSPACE-complete. Theorem \ref{thm:Toggle_Pspace} simply supports the notion of the richness of Toggle and its ability to simulate other PSPACE-complete problems via logspace reductions.
\end{remark}



\newpage
\bibliographystyle{plain}
\bibliography{bibliography}

\begin{thebibliography}{99}

\bibitem{AF98} M. Anderson and T. Feil, Turning Lights Out with Linear Algebra, \textit{Mathematics Magazine}, \textbf{71(4)} (1998), 300--303.

\bibitem{BCG}
E.~R.~Berlekamp, J.~H.~Conway, and R.~K.~Guy, Winning Ways for Your Mathematical Plays, \textit{A.~K.~Peters} \textbf{1} (2001).

\bibitem{BK}
H.~L.~Bodlaender and D.~Kratsch, Kayles and Nimbers, \emph{J. Algorithms} \textbf{43} (2002), 106--119.

\bibitem{BDF}  
S.\ Brown, S.\ Daugherty, E.\ Fiorini, B.\ Maldonado, D.\ Manzano-Ruiz, S.\ Rainville, R.\ Waechter, T.~W.~H.\ Wong, Nimber sequences of Node-Kayles games, \emph{J. of Integer Sequences} \textbf{23} (2020), 20.3.5.

\bibitem{DD00}
D.~Du and K.~Ko, Theory of Computational Complexity, \textit{Wiley Online Books}, (2000).

\bibitem{ET75}S.~Even and R.~E.~Tarjan, A combinatorial problem which is complete in polynomial space, \textit{Seventh Annual ACM Symposium on Theory of Computing}, (1975), 66-71.

\bibitem{LK}
P.~Lemke and D.~J.~Kleitman, \textit{An Addition Theorem on the Integers Modulo $n$}, \textit{J. of Number Theory} \textbf{31}, (1989), 335--345.

\bibitem{TS78}
T. Schaefer, On the complexity of some two-person perfect-information games, \textit{J. of Computer and System Sciences}, \textbf{16(2)}, (1978), 185-225.

\bibitem{oeis}
N.~J.~A.~Sloane, \textit{The On-Line Encyclopedia of Integer Sequences}, \url{http://oeis.org}.

\bibitem{SS09}
A.~Steimle and W.~Staton, The Isomorphism Classes of the Generalized Petersen Graphs, \textit{Discret. Math.} \textbf{309(1)} (2009), 231-237.

\bibitem{LS77}
L. Stockmeyer, The polynomial-time hierarchy, \textit{Theoretical Computer Science},
\textbf{3(1)}, (1977), 1-22.






 



  





 







\newpage
\end{thebibliography}

\section*{Appendix A: Computer Code}\label{code}

Below we include the Matlab and CGSuite code that was used to construct the Tables in Appendix B.

The following code is written in Matlab.

\begin{lstlisting}[language=Matlab]
function [mex] = min_excluded(s)
nats=0:max(s)+1;
mex=min(setdiff(nats,s));
return;
end
\end{lstlisting}

\begin{lstlisting}[language=Matlab]
function [nim] = nim_petersen_n1(n)
% Calculates the nimber of GP(n,1) with all ones

% Set base cases
Hvals=[0,0,0,0,1,1,1]; % Nimbers of H_n from n=1 to 7
% H_1, H_2, and H_4 are all zeros
%
% H_3 (weird case):
% 0 1 0
% 0 1 0
% 
% H_5:
% 0 0 1 0 0
% 0 1 1 1 0
%
% H_7:
% 0 0 1 1 1 0 0
% 0 1 1 1 1 1 0

Dvals=[0,0,0,0,1,1,1]; % Nimbers of D_n from n=1 to 7
% D_1, D_2, and D_3 are all zeros
%
% D_4:
% 0 0 1 0
% 0 1 0 0
%
% D_6:
% 0 0 1 1 1 0
% 0 1 1 1 0 0

% For n within base cases, Nimber of GP(n,1) is equal to mex of
% the Nimber of H_(n+1)
if n < length(Hvals) 
    nim=min_excluded(Hvals(n+1));
    return;
end

% Compute Nimbers of H_i and D_i for all i<=n
for i=8:n+1
    
    % Make list of all possible moves on D_i
    D_allmoves=1:i-4; % Set size of list

    % Ways to split D_i into a D and an H
    for j=1:i-4
        % Set each element of list to the Nimber of the graph with 
        % components H_(j+2) and D_(i+1-(j+2))
        D_allmoves(j)=bitxor(Hvals(j+2), Dvals(i+1-(j+2)));
    end
    
    % Nimber of D_i is mex of the Nimbers of all next states
    Dvals(i)=min_excluded(D_allmoves);
    

    % Make list of all possible moves on H_i
    H_allmoves=1:2*floor((i+1)/2)-4; % Set size of list
    
    % Ways to split H_i into 2 H's (by playing on top row)
    for j=1:floor((i+1)/2)-2
        % Set each element of list to the Nimber of the graph with
        % components H_(j+2) and H_(i+1-(j+2))
        H_allmoves(j)=bitxor(Hvals(j+2), Hvals(i+1-(j+2)));
    end

    % Ways to split H_i into 2 D's (by playing on bottom row)
    for j=1:floor((i+1)/2)-2
        % Set each element of list to the Nimber of the graph with
        % components D_(j+2) and D_(i+1-(j+2))
        H_allmoves(j+floor((i+1)/2)-2)=bitxor(Dvals(j+2), ...
        Dvals(i+1-(j+2)));
    end
    
    % Nimber of H_i is mex of Nimbers of all next states
    Hvals(i)=min_excluded(H_allmoves);
end

% Nimber of GP(n,1) with all ones is the mex of the Nimber of H_(n+1)
nim=min_excluded(Hvals(n+1));
return;

end
\end{lstlisting}

The following code is written in CGSuite.

\begin{footnotesize}
\begin{lstlisting}
 //To get nimber: Petersenk("111|000",2).CanonicalForm
 //returns nimber of GP(3,2) with 1s on outside, 0s inside
class Petersenk extends ImpartialGame
    var nodes; //Grid object representing nodes (in the form "000|111")
    var n; //n value of GP(n,k)
    var k; //k value of GP(n,k)
    //initialize graph. Parameter 'a' can be a Grid object, or a string of
    //the form "000|111" where the left side represents the outer cycle and
    //the right side represents the inner cycle(s), or a number representing
    //the n value for GP(n,k). The parameter 'b' is the k value for GP(n,k).
    method Petersenk(a,b)
        //if 'a' is a Grid, set nodes to be 'a' and set n to be the number
        //of columns in 'a'
        if a is Grid then
            nodes:=a;
            n:=nodes.ColumnCount;
        //if 'a' is a string, set nodes to a Grid object formed by 'a'.
        //Set n to be the number of columns in the grid.
        elseif a is String then
            nodes:=Grid.ParseGrid(a,"01");
            n:=nodes.ColumnCount;
        //If 'a' is a number, initialize graph with 0's on all outer nodes
        //and 1's on all inner nodes.
        else
            //first row is outer cycle, second row is inner cycle(s)
            nodes:=Grid(2,a);
            for i from 1 to a do
                nodes[1,i]:=0;
                nodes[2,i]:=1;
            end
        end
        k:=b;
    end
    //Finds all possible game states that can be created by playing one move
    //on the current game state.
    override method Options(Player player)
        options:=[];
        //loop through all columns
        for i from 1 to n do
            //create a new game state by moving on the ith outer node
            if nodes[1,i] == 1 then
                j:=1;
                newstate1:=nodes;
                //flip this node
                newstate1[1,i]:=(newstate1[1,i]+1)%2;
                //flip connecting node in inner cycle
                newstate1[2,i]:=(newstate1[2,i]+1)%2;
                //flip node to the left in outer cycle
                if i == 1 then
                    newstate1[j,n]:=(newstate1[j,n]+1)%2;
                end
                if i != 1 then
                    newstate1[j,i-1]:=(newstate1[j,i-1]+1)%2;
                end
                //flip node to the right in outer cycle
                if i == n then
                    newstate1[j,1]:=(newstate1[j,1]+1)%2;
                end
                if i != n then
                    newstate1[j,i+1]:=(newstate1[j,i+1]+1)%2;
                end
                //check if this move is legal by checking if the number of zeros
                //has increased. If so, add the new state to the list of options
                if numZeros(newstate1) > numZeros(nodes) then
                    options.Add(Petersenk(newstate1,k));
                end
            end
            //create a new game state by moving on the ith inner node
            if nodes[2,i] == 1 then
                j:=2;
                newstate2:=nodes;
                //flip this node
                newstate2[1,i]:=(newstate2[1,i]+1)%2;
                //flip connecting node in outer cycle
                newstate2[2,i]:=(newstate2[2,i]+1)%2;
                //flip node to the left in inner cycle
                if i-k < 1 then
                    newstate2[j,n+(i-k)]:=(newstate2[j,n+(i-k)]+1)%2;
                else
                    newstate2[j,i-k]:=(newstate2[j,i-k]+1)%2;
                end
                //flip node to the right in inner cycle
                if i+k > n then
                    newstate2[j,i+k-n]:=(newstate2[j,i+k-n]+1)%2;
                else
                    newstate2[j,i+k]:=(newstate2[j,i+k]+1)%2;
                end
                //check if this move is legal by checking if the number of zeros
                //has increased. If so, add the new state to the list of options
                if numZeros(newstate2) > numZeros(nodes) then
                    options.Add(Petersenk(newstate2,k));
                end
            end
        end
        //return the list of possible next states
        return options;
    end
    //returns the current game state as a Grid (such as "000|111") where
    //the left side represents the outer cycle and the right side represents the
    //inner cycle
    method getGrid()
        return nodes;
    end
    //calculates the number of nodes with value zero in the current game state
    method numZeros(g)
        count:=0;
        for i from 1 to g.RowCount do
            for j from 1 to g.ColumnCount do
                if g[i,j]==0 then
                    count:=count+1;
                end
            end
        end
        return count;
    end
end
\end{lstlisting}
\end{footnotesize}

\newpage \section*{Appendix B: Tables}\label{tables}

\begin{center}
    \begin{longtable}{c|cccccccccccc}
    \caption{Nimbers for $\mathcal{P}_{0,1}(n,k)$}
    \label{tab:NimbersP01}\\

    \hline \multicolumn{1}{c|}{$n\backslash k$} & \multicolumn{1}{c}{1} & \multicolumn{1}{c}{2}  & \multicolumn{1}{c}{3} & \multicolumn{1}{c}{4}  & \multicolumn{1}{c}{5} & \multicolumn{1}{c}{6}  & \multicolumn{1}{c}{7} & \multicolumn{1}{c}{8}  & \multicolumn{1}{c}{9} & \multicolumn{1}{c}{10}  & \multicolumn{1}{c}{11} & \multicolumn{1}{c}{12} \\ 
    \hline
    \endfirsthead

    \multicolumn{13}{c}{{\bfseries \tablename\ \thetable{} -- continued from previous page}} \\ \hline \multicolumn{1}{c|}{$n\backslash k$} & \multicolumn{1}{c}{1} & \multicolumn{1}{c}{2}  & \multicolumn{1}{c}{3} & \multicolumn{1}{c}{4}  & \multicolumn{1}{c}{5} & \multicolumn{1}{c}{6}  & \multicolumn{1}{c}{7} & \multicolumn{1}{c}{8}  & \multicolumn{1}{c}{9} & \multicolumn{1}{c}{10}  & \multicolumn{1}{c}{11} & \multicolumn{1}{c}{12} \\ \hline
    \endhead

    \hline \multicolumn{13}{r}{{Continued on next page}} \\ \hline
    \endfoot

    \hline
    \endlastfoot

3 & 1 \\ \hline
4 & 1 \\ \hline
5 & 1 & 1 \\ \hline
6 & 0 & 0 \\ \hline
7 & 0 & 0 & 0 \\ \hline
8 & 0 & 0 & 0 \\ \hline
9 & 0 & 0 & 0 & 0 \\ \hline
10 & 0 & 0 & 0 & 0 \\ \hline
11 & 1 & 1 & 0 & 0 & 1 \\ \hline
12 & 0 & 0 & 1 & 0 & 0 \\ \hline
13 & 0 & 0 & 1 & 0 & 0 & 0 \\ \hline
14 & 0 & 0 & 0 & 0 & 0 & 0 \\ \hline
15 & 0 & 0 & 0 & 1 & 0 & 0 & 0 \\ \hline
16 & 0 & 0 & 0 & 0 & 0 & 0 & 0 \\ \hline
17 & 1 & 1 & 0 & 0 & 1 & 0 & 0 & 1 \\ \hline
18 & 0 & 0 & 0 & 0 & 0 & 0 & 0 & 0 \\ \hline 
19 & 0 & 0 & 0 & 0 & 0 & 0 & 0 & 0 & 0 \\ \hline 
20 & 0 & 0 & 0 & 0 & 0 & 0 & 0 & 0 & 0 \\ \hline 
21 & 0 & 0 & 1 & 1 & 0 & 1 & 0 & 0 & 0 & 0 \\ \hline
22 & 0 & 0 & 0 & 0 & 0 & 0 & 0 & 0 & 0 & 0 \\ \hline
23 & 0 & 0 & 0 & 1 & 0 & 0 & 0 & 1 & 0 & 0 & 0 \\ \hline
24 & 0 & 0 & 0 & 0 & 0 & 0 & 0 & 0 & 0 & 0 & 0 \\ \hline
25 & 0 & 0 & 0 & 0 & 1 & 0 & 0 & 0 & 0 & 0 & 0 \\ \hline
 
    \end{longtable}
\end{center}
\newpage
\begin{center}
    \begin{longtable}{c|cccccccccccc}
    \caption{Nimbers for $\mathcal{P}_{1,0}(n,k)$}
    \label{tab:NimbersP10}\\

    \hline \multicolumn{1}{c|}{$n\backslash k$} & \multicolumn{1}{c}{1} & \multicolumn{1}{c}{2}  & \multicolumn{1}{c}{3} & \multicolumn{1}{c}{4}  & \multicolumn{1}{c}{5} & \multicolumn{1}{c}{6}  & \multicolumn{1}{c}{7} & \multicolumn{1}{c}{8}  & \multicolumn{1}{c}{9} & \multicolumn{1}{c}{10}  & \multicolumn{1}{c}{11} & \multicolumn{1}{c}{12} \\ 
    \hline
    \endfirsthead

    \multicolumn{13}{c}{{\bfseries \tablename\ \thetable{} -- continued from previous page}} \\ \hline \multicolumn{1}{c|}{$n\backslash k$} & \multicolumn{1}{c}{1} & \multicolumn{1}{c}{2}  & \multicolumn{1}{c}{3} & \multicolumn{1}{c}{4}  & \multicolumn{1}{c}{5} & \multicolumn{1}{c}{6}  & \multicolumn{1}{c}{7} & \multicolumn{1}{c}{8}  & \multicolumn{1}{c}{9} & \multicolumn{1}{c}{10}  & \multicolumn{1}{c}{11} & \multicolumn{1}{c}{12} \\ \hline
    \endhead

    \hline \multicolumn{13}{r}{{Continued on next page}} \\ \hline
    \endfoot

    \hline
    \endlastfoot

3 & 1 \\ \hline
4 & 1 \\ \hline
5 & 1 & 1 \\ \hline
6 & 0 & 0 \\ \hline
7 & 0 & 0 & 0 \\ \hline
8 & 0 & 0 & 0 \\ \hline
9 & 0 & 0 & 1 & 0 \\ \hline
10 & 0 & 0 & 0 & 0 \\ \hline
11 & 1 & 1 & 0 & 0 & 1 \\ \hline
12 & 0 & 0 & 1 & 0 & 0 \\ \hline
13 & 0 & 0 & 0 & 1 & 0 & 0 \\ \hline
14 & 0 & 0 & 0 & 0 & 0 & 0 \\ \hline
15 & 0 & 0 & 1 & 1 & 0 & 1 & 0 \\ \hline
16 & 0 & 0 & 0 & 0 & 0 & 0 & 0 \\ \hline
17 & 1 & 1 & 0 & 0 & 0 & 0 & 1 & 1 \\ \hline
18 & 0 & 0 & 0 & 0 & 0 & 0 & 0 & 0 \\ \hline 
19 & 0 & 0 & 0 & 0 & 0 & 0 & 0 & 0 & 0 \\ \hline 
20 & 0 & 0 & 0 & 0 & 0 & 0 & 0 & 0 & 0 \\ \hline 
21 & 0 & 0 & 0 & 0 & 1 & 0 & 1 & 0 & 0 & 0 \\ \hline
22 & 0 & 0 & 0 & 0 & 0 & 0 & 0 & 0 & 0 & 0 \\ \hline
23 & 0 & 0 & 1 & 0 & 0 & 1 & 0 & 0 & 0 & 0 & 0 \\ \hline
24 & 0 & 0 & 0 & 0 & 0 & 0 & 0 & 0 & 0 & 0 & 0 \\ \hline
25 & 0 & 0 & 0 & 0 & 0 & 0 & 0 & 0 & 0 & 0 & 0 \\ \hline
 
    \end{longtable}
\end{center}

\begin{center}
    \begin{longtable}{c|ccccccc}
    \caption{Nimbers for $\mathcal{P}_{1,1}(n,k)$}
    \label{tab:NimbersP11}\\

    \hline \multicolumn{1}{c|}{$n\backslash k$} & \multicolumn{1}{c}{1} & \multicolumn{1}{c}{2}  & \multicolumn{1}{c}{3} & \multicolumn{1}{c}{4}  & \multicolumn{1}{c}{5} & \multicolumn{1}{c}{6}  & \multicolumn{1}{c}{7} \\ 
    \hline
    \endfirsthead

    \multicolumn{8}{c}{{\bfseries \tablename\ \thetable{} -- continued}} \\ \hline \multicolumn{1}{c|}{$n\backslash k$} & \multicolumn{1}{c}{1} & \multicolumn{1}{c}{2}  & \multicolumn{1}{c}{3} & \multicolumn{1}{c}{4}  & \multicolumn{1}{c}{5} & \multicolumn{1}{c}{6}  & \multicolumn{1}{c}{7} \\ \hline
    \endhead

    \hline \multicolumn{8}{r}{{Continued on next page}} \\ \hline
    \endfoot

    \hline
    \endlastfoot

3 & 1 &  &  &  &  &  &  \\ \hline
4 & 0 &  &  &  &  &  &  \\ \hline
5 & 0 & 1 &  &  &  &  &  \\ \hline
6 & 0 & 0 &  &  &  &  &  \\ \hline
7 & 1 & 2 & 2 &  &  &  &  \\ \hline
8 & 0 & 1 & 0 &  &  &  &  \\ \hline
9 & 0 & 1 & 0 & 1 &  &  & \\ \hline
10 & 0 & 0 & 0 & 0 &  &  &  \\ \hline
11 & 1 & 0 & 1 & 1 & 0 &  &  \\ \hline
12 & 0 & 0 & 0 & 0 & 0 &  &  \\ \hline
13 & 0 & 0 & 1 & 1 & 1 & 0 &  \\ \hline
14 & 0 & 0 & 0 & 0 & 0 & 0 &  \\ \hline
15 & 0 & 2 & 1 & 0 & 1 & 0 & 2 \\ \hline
16 & 0 & 0 & 0 & 1 & 0 & 0 & 0 \\ \hline
17 & 1 & 1 & 1 & 0 & 0 & ? & ? \\ \hline
18 & 0 & 0 & 0 & 0 & 0 & ? & ? \\ \hline 
19 & ? & 1 & 1 & 0 & 0 & ? & ? \\ \hline 
20 & ? & ? & 0 & ? & 0 & ? & ? \\ \hline 
21 & ? & ? & 1 & ? & 0 & ? & ? \\ \hline
23 & ? & 0 & ? & ? & ? & ? & ? \\ \hline
 
    \end{longtable}
\end{center}

\begin{center}
    \begin{longtable}{c|ccccccccccc}
    \caption{Nimbers for $\mathcal{P}(2m,2)$ and $\mathcal{P}(2m+1,2)$}
    \label{tab:NimbersP2m2}\\

    \hline \multicolumn{1}{c|}{$m$} & \multicolumn{1}{c}{3} & \multicolumn{1}{c}{4}  & \multicolumn{1}{c}{5} & \multicolumn{1}{c}{6}  & \multicolumn{1}{c}{7} & \multicolumn{1}{c}{8}  & \multicolumn{1}{c}{9} & \multicolumn{1}{c}{10}  & \multicolumn{1}{c}{11} & \multicolumn{1}{c}{12} & \multicolumn{1}{c}{13} \\ 
    \hline
    \endfirsthead

    \multicolumn{12}{c}{{\bfseries \tablename\ \thetable{} -- continued from previous page}} \\ \hline \multicolumn{1}{c|}{$m$} & \multicolumn{1}{c}{3} & \multicolumn{1}{c}{4}  & \multicolumn{1}{c}{5} & \multicolumn{1}{c}{6}  & \multicolumn{1}{c}{7} & \multicolumn{1}{c}{8}  & \multicolumn{1}{c}{9} & \multicolumn{1}{c}{10}  & \multicolumn{1}{c}{11} & \multicolumn{1}{c}{12} & \multicolumn{1}{c}{13} \\ \hline
    \endhead

    \hline \multicolumn{12}{r}{{Continued on next page}} \\ \hline
    \endfoot

    \hline
    \endlastfoot

$\mathcal{P}_{0,1}(2m,2)$ & 0 & 0 & 0 & 0 & 0 & 0 & 0 & 0 & 0 & 0 & 0 \\ \hline
$\mathcal{P}_{1,0}(2m,2)$ & 0 & 0 & 0 & 0 & 0 & 0 & 0 & 0 & 0 & 0 & 0 \\ \hline
$\mathcal{P}_{0,1}(2m+1,2)$ & 0 & 0 & 1 & 0 & 0 & 1 & 0 & 0 & 0 & 0 & 1 \\ \hline
$\mathcal{P}_{1,0}(2m+1,2)$ & 0 & 0 & 1 & 0 & 0 & 1 & 0 & 0 & 0 & 0 & 1 \\ \hline
 
    \end{longtable}
\end{center}

\begin{center}
    \begin{longtable}{c|ccccccccc}
    \caption{Nimbers for $\mathcal{P}_{0,1}(3m,3)$ and $\mathcal{P}_{1,0}(3m,3)$}
    \label{tab:NimbersP3m3}\\

    \hline \multicolumn{1}{c|}{$m$} & \multicolumn{1}{c}{3} & \multicolumn{1}{c}{4}  & \multicolumn{1}{c}{5} & \multicolumn{1}{c}{6}  & \multicolumn{1}{c}{7} & \multicolumn{1}{c}{8}  & \multicolumn{1}{c}{9} & \multicolumn{1}{c}{10}  & \multicolumn{1}{c}{11} \\ 
    \hline
    \endfirsthead

    \multicolumn{10}{c}{{\bfseries \tablename\ \thetable{} -- continued from previous page}} \\ \hline \multicolumn{1}{c|}{$m$} & \multicolumn{1}{c}{3} & \multicolumn{1}{c}{4}  & \multicolumn{1}{c}{5} & \multicolumn{1}{c}{6}  & \multicolumn{1}{c}{7} & \multicolumn{1}{c}{8}  & \multicolumn{1}{c}{9} & \multicolumn{1}{c}{10}  & \multicolumn{1}{c}{11} \\ \hline
    \endhead

    \hline \multicolumn{10}{r}{{Continued on next page}} \\ \hline
    \endfoot

    \hline
    \endlastfoot

$\mathcal{P}_{0,1}(3m,3)$ & 0 & 1 & 0 & 0 & 1 & 0 & 0 & 0 & 0 \\ \hline
$\mathcal{P}_{1,0}(3m,3)$ & 1 & 1 & 1 & 0 & 0 & 0 & 0 & 0 &  \\ \hline
 
    \end{longtable}
\end{center}

\begin{center}
    \begin{longtable}{c|ccccccccccccc}
    \caption{Nimbers for $\mathcal{P}_{0,1}(3k,k)$ and $\mathcal{P}_{1,0}(3k,k)$}
    \label{tab:NimbersP3kk}\\

    \hline \multicolumn{1}{c|}{$k$} & \multicolumn{1}{c}{1} & \multicolumn{1}{c}{2} & \multicolumn{1}{c}{3} & \multicolumn{1}{c}{4} & \multicolumn{1}{c}{5} & \multicolumn{1}{c}{6}  & \multicolumn{1}{c}{7} & \multicolumn{1}{c}{8} & \multicolumn{1}{c}{9} & \multicolumn{1}{c}{10} & \multicolumn{1}{c}{11} & \multicolumn{1}{c}{12} & \multicolumn{1}{c}{13} \\ 
    \hline
    \endfirsthead

    \multicolumn{14}{c}{{\bfseries \tablename\ \thetable{} -- continued from previous page}} \\ \hline \multicolumn{1}{c|}{$k$} & \multicolumn{1}{c}{1} & \multicolumn{1}{c}{2} & \multicolumn{1}{c}{3} & \multicolumn{1}{c}{4} & \multicolumn{1}{c}{5} & \multicolumn{1}{c}{6}  & \multicolumn{1}{c}{7} & \multicolumn{1}{c}{8} & \multicolumn{1}{c}{9} & \multicolumn{1}{c}{10} & \multicolumn{1}{c}{11} & \multicolumn{1}{c}{12} & \multicolumn{1}{c}{13} \\ \hline
    \endhead

    \hline \multicolumn{14}{r}{{Continued on next page}} \\ \hline
    \endfoot

    \hline
    \endlastfoot

$\mathcal{P}_{0,1}(3k,k)$ & 1 & 0 & 0 & 0 & 0 & 0 & 0 & 0 & 0 & 0 & 0 & ? & 0 \\ \hline
$\mathcal{P}_{1,0}(3k,k)$ & 1 & 0 & 1 & 0 & 0 & 0 & 1 & 0 & 1 & 0 & 1 &  &  \\ \hline
 
    \end{longtable}
\end{center}

\begin{center}
    \begin{longtable}{c|ccccccccccc}
    \caption{Nimbers for $\mathcal{P}(3m+\varepsilon,3)$, $\varepsilon = 1,2$}
    \label{tab:NimbersP3m+i3}\\

    \hline \multicolumn{1}{c|}{$m$} & \multicolumn{1}{c}{1} & \multicolumn{1}{c}{2}  & \multicolumn{1}{c}{3} & \multicolumn{1}{c}{4}  & \multicolumn{1}{c}{5} & \multicolumn{1}{c}{6}  & \multicolumn{1}{c}{7} & \multicolumn{1}{c}{8} \\ 
    \hline
    \endfirsthead

    \multicolumn{9}{c}{{\bfseries \tablename\ \thetable{} -- continued from previous page}} \\ \hline \multicolumn{1}{c|}{$m$} & \multicolumn{1}{c}{1} & \multicolumn{1}{c}{2}  & \multicolumn{1}{c}{3} & \multicolumn{1}{c}{4}  & \multicolumn{1}{c}{5} & \multicolumn{1}{c}{6}  & \multicolumn{1}{c}{7} & \multicolumn{1}{c}{8} \\ \hline
    \endhead

    \hline \multicolumn{9}{r}{{Continued on next page}} \\ \hline
    \endfoot

    \hline
    \endlastfoot

$\mathcal{P}_{0,1}(3m+1,3)$ & 1 & 0 & 0 & 0 & 0 & 0 & 0 &  \\ \hline
$\mathcal{P}_{1,0}(3m+1,3)$ & 1 & 0 & 0 & 1 & 0 & 0 & 0 &  \\ \hline
$\mathcal{P}_{0,1}(3m+2,3)$ & 1 & 0 & 0 & 0 & 0 & 0 & 0 & 0 \\ \hline
$\mathcal{P}_{1,0}(3m+2,3)$ & 1 & 0 & 0 & 0 & 0 & 0 & 1 & 0 \\ \hline
 
    \end{longtable}
\end{center}

\section*{Statements and Declarations}
\begin{itemize}
    \item The authors declare the following funding source: NSF DMS-1852378 and NSF DMS-2150299.
    
    \item The authors have no relevant financial or non-financial interests to disclose.
    
    \item All authors contributed to the study conception and design. Material preparation, data collection and analysis were performed by all authors. All authors read and approved the final manuscript.

    \item The authors declare that they have no conflict of interest.

    \item Data sharing is not applicable to this article.
\end{itemize}

\end{document}